\documentclass[a4paper,11pt, reqno]{amsart}
\usepackage{amsfonts}
\usepackage{latexsym}
\usepackage{amssymb}
\usepackage{amsmath}
\usepackage{color}
\usepackage{bbm}
\usepackage{tikz}
\usepackage{enumerate}
\usepackage{mathrsfs}
\usepackage{todonotes}

\usepackage[left=2.5cm, top=2.5cm,bottom=2.5cm,right=2.5cm]{geometry}

\newcommand{\R}{\mathbb R}
\newcommand{\N}{\mathbb N}
\newcommand{\C}{\mathbb C}
\newcommand{\GG}{\mathbb G}
\newcommand{\E}{\mathbb E}
\newcommand{\Pro}{\mathbb P}
\newcommand{\dif}{\,\mathrm{d}}

\newcommand{\Var}{\mathrm{Var}}
\newcommand{\Cov}{\mathrm{Cov}}

\def\dint{\textup{d}}
\newcommand{\SSS}{\ensuremath{{\mathbb S}}}

\newcommand{\B}{\ensuremath{{\mathbb B}}}

\newtheorem{thm}{Theorem}[section]

\newtheorem{lemma}[thm]{Lemma}
\newtheorem{df}[thm]{Definition}
\newtheorem{proposition}[thm]{Proposition}

\newtheorem{rmk}[thm]{Remark}

\def\bC{\mathbf{C}}

\def\bP{\mathbf{P}}

\def\bU{\mathbf{U}}

\def\bW{\mathbf{W}}

\usepackage{nomencl}
\makenomenclature

\allowdisplaybreaks

\begin{document}


\title[]{Gaussian fluctuations for high-dimensional random projections of $\ell_p^n$-balls}

\author[D. Alonso-Guti\'errez]{David Alonso-Guti\'errez}
\address{Departamento de Matem\'aticas, Universidad de Zaragoza, Spain} \email{alonsod@unizar.es}

\author[J. Prochno]{Joscha Prochno}
\address{School of Mathematics \& Physical Sciences, University of Hull, United Kingdom} \email{j.prochno@hull.ac.uk}

\author[C. Th\"ale]{Christoph Th\"ale}
\address{Faculty of Mathematics, Ruhr University Bochum, Germany} \email{christoph.thaele@rub.de}




\begin{abstract}
In this paper, we study high-dimensional random projections of $\ell_p^n$-balls. More precisely, for any $n\in\N$ let $E_n$ be a random subspace of dimension $k_n\in\{1,\ldots,n\}$ and $X_n$ be a random point in the unit ball of $\ell_p^n$. Our work provides a description of the Gaussian fluctuations of the Euclidean norm $\|P_{E_n}X_n\|_2$ of random orthogonal projections of $X_n$ onto $E_n$. In particular, under the condition that $k_n\to\infty$ it is shown that these random variables satisfy a central limit theorem, as the space dimension $n$ tends to infinity. Moreover, if $k_n\to\infty$ fast enough, we provide a Berry-Esseen bound on the rate of convergence in the central limit theorem. At the end we provide a discussion of the large deviations counterpart to our central limit theorem.
\end{abstract}

\maketitle

\tableofcontents

\section{Introduction and results}

The study of high-dimensional phenomena and, in particular, the description of geometric properties of high-dimensional convex bodies is what is known today as asymptotic geometric analysis. In this branch of mathematics analysis, geometry, and probability intertwine in a highly non-trivial way. It has become clear that the presence of high dimensions forces a certain regularity in the geometry of convex bodies in the same way in which the presence of high dimensions forces a regularity in the behavior of random vectors. One instance is the central limit theorem, which is widely known in probability theory to capture the fluctuations of sums of (independent) random variables. This theorem has a geometric counterpart. It was proved by Klartag \cite{KlartagCLT, KlartagCLT2} that most $k$-dimensional marginals of a random vector uniformly distributed in an isotropic convex body are approximately Gaussian, provided that $k=k_n$ is smaller than $n^\kappa$ for some absolute constant $\kappa\in (0,1)$ which is known to satisfy $\kappa< 1/14$. Furthermore, he obtained a rate of convergence in the total variation distance. Let us mention that in the case of $1$-unconditional isotropic convex bodies the value of $\kappa$ was improved by M. Meckes \cite{MeckesM12} to $\kappa<1/7$. In this context let us also refer to another work of Klartag \cite{KlartagUnconditional} for an optimal rate of convergence for such bodies in the so-called Kolmogorov distance when $k=1$ (see also the detailed discussion at the end of \cite{MeckesM12}).

Besides the $k$-dimensional marginals of a random vector, only few random geometric parameters associated to convex bodies in high dimensions have been shown to satisfy a central limit theorem. In \cite{PPZ14}, Paouris, Pivovarov and Zinn have proved the central limit behavior for the volume of $k$-dimensional random projections of the $n$-dimensional cube (in this set-up, $k$ was not allowed to vary with $n$). When taking $k=1$, their result turns into a central limit theorem for a random projection of the $1$-norm $\Vert\theta\Vert_1$, where $\theta$ is a random vector uniformly distributed on the Euclidean sphere $\SSS^{n-1}$ (by a different method this has also been obtained in \cite[Theorem 3.6]{KLZ15}). This particular case was recently extended in \cite{KPT17} to a central limit theorem for arbitrary $\ell_p^n$-balls $\B_p^n$ with $1<p\leq\infty$. This is a consequence of a multivariate central limit theorem that the authors proved in \cite{KPT17} (all notions and notation will be introduced in Section \ref{sec:PrelimNotation} below). Moreover, central and non-central limit theorems for the volume of random simplices in high dimensions have recently been studied in \cite{GroteKabluchkoThaele}.

While the central limit theorem underlines the universal behavior of Gaussian fluctuations, it is widely known in probability theory that the large  deviation behavior, which deals with probabilities far beyond the scale of the central limit theorem, is much more sensitive to the involved random variables. However, it was only recently that large deviation principles (LDP) for random vectors uniformly distributed on convex bodies have been studied in order to access non-universal features and unveil properties that distinguish between different convex bodies. In \cite{GKR}, Gantert, Kim and Ramanan proved an LDP for 1-dimensional projections of random vectors uniformly distributed in the $\ell_p^n$-ball. In the annealed case, this result was extended in \cite{APT16} to a higher-dimensional setting, showing that the Euclidean norm of the projection of a random vector uniformly distributed in $\B_p^n$ onto a random subspace satisfies an LDP. Our main result, Theorem \ref{thm:CLTEuclideanNormProjections} below, complements these findings on the limiting behavior of the Euclidean norm of random projections. More precisely, we prove a normal fluctuations counterpart, that is, we show that the Euclidean norm of such random orthogonal projections satisfies a central limit theorem, as the space dimension $n$ tends to infinity. Let us remark that in our set-up, where the vectors and subspaces are chosen simultaneously at random and only for the special case of the uniform distribution on $\B_p^n$ this is a direct consequence of the central limit theorems of Klartag or M. Meckes, provided that the subspace dimensions $k_n$ tend to infinity and satisfy $k_n<n^\kappa$ with $\kappa<1/7$ (see Remark \ref{rem:KlartagCLT} below). However, this is not the case for the other probability measures on $\B_p^n$ we consider and also not when the subspace dimensions $k_n$ grow faster with $n$. With this paper, we provide a central limit theorem for the Euclidean norm of random projections of random vectors distributed on $\B_p^n$ in the full regime where $k_n\to\infty$, as $n\to\infty$, while $k_n/n\to\lambda\in[0,1]$. In addition to that, when the subspace dimensions grow faster than $n^{2/3}$, we are able to provide a Berry-Esseen type rate of convergence. Let us point out that for a fixed 1-dimesional subspace, a Berry-Esseen estimate was proved in \cite{GS} for a standarized projection.

In addition and opposed to \cite{APT16,GKR}, our central limit theorem will describe the Gaussian fluctuations of a whole family of probability distributions on $\B_p^n$ that has been introduced in the paper \cite{BartheGuedonEtAl} of Barthe, Gu\'edon, Mendelson and Naor. As a special case, this class contains the uniform distribution considered in \cite{APT16,GKR} as well as the cone probability measure on $\B_p^n$ (compare with the discussion below). To introduce these distributions, for $1\leq p<\infty$, we let $\bW$ be any Borel probability measure on $[0,\infty)$, $\bU_{n,p}$ be the uniform distribution  and $\bC_{n,p}$ stand for the cone probability measure on $\B_p^n$. The distributions we consider are of the form
\begin{equation}\label{eq:DefMeasurePnpW}
\bP_{n,p,\bW} := \bW(\{0\})\,\bC_{n,p} + H\,\bU_{n,p},
\end{equation}
where the function $H:\B_p^n\to\R$ is given by $H(x)=h(\|x\|_p)$ with
$$
h(r) = {1\over\Gamma\big(1+{n\over p}\big)}{1\over (1-r^p)^{1+n/p}}\int_0^\infty s^{n/p}e^{-r^ps/(1-r^p)}\,\bW(\dint s),\qquad r\in[0,1]
$$
(note that a factor $2^n$ is missing in the statement of \cite[Theorem 3]{BartheGuedonEtAl}). The class of measures of the form $\bP_{n,p,\bW}$ contains the following important cases, which are of particular interest (see Theorem 1, Theorem 3, Corollary 3 and Corollary 4 in \cite{BartheGuedonEtAl}):
\begin{itemize}
\item[(i)] If $\bW$ is the exponential distribution with mean $1$, then $\bW(\{0\})=0$, $H\equiv 1$ and $\bP_{n,p,\bW}$ reduces to the uniform distribution $\bU_{n,p}$ on $\B_p^n$.
\item[(ii)] If $\bW=\delta_0$ is the Dirac measure concentrated at $0$, then $\bW(\{0\})=1$, $H\equiv 0$ and $\bP_{n,p,\bW}$ is just the cone probability measure on $\B_p^n$.
\item[(iii)] If $\bW={\rm Gamma}(\alpha,1)$ is a gamma distribution with shape parameter $\alpha>0$ and rate $1$, then $\bP_{n,p,\bW}$ is the beta-type probability measure on $\B_p^n$ with density given by
$$
x\mapsto {\Gamma\big(\alpha+{n\over p}\big)\over\Gamma(\alpha)\Big(2\Gamma\big(1+{1\over p}\big)\Big)^n}\,\big(1-\|x\|_p^p\big)^{\alpha-1},\qquad x\in\B_p^n\,.
$$
In particular, if $\alpha=m/p$ for some $m\in\N$, this is the image of the cone probability measure $\bC_{n+m,p}$ on $\B_p^{n+m}$ under the orthogonal projection onto the first $n$ coordinates. Similarly, if $\alpha=1+m/p$, this is the image of the uniform distribution $\bU_{n+m,p}$ on $\B_p^n$ under the same projection.
\end{itemize}

We are now prepared to present our main results. Let us denote by $\GG_{n,k}$ the Grassmannian of $k$ dimensional subspaces of $\R^n$ equipped with the Haar probability measure $\nu_{n,k}$ and for $E\in \GG_{n,k}$ write $P_E$ for the orthogonal projection onto $E$.

\begin{thm}\label{thm:CLTEuclideanNormProjections}
Let $1\leq p<\infty$ and $\bW$ be a probability distribution on $[0,\infty)$. Further, let $(k_n)_{n\in\N}$ be a sequence in $\N$ with $k_n\in\{1,\dots,n\}$, $(X_n)_{n\in\N}$ be a sequence of independent random vectors distributed in $\B_p^n$ according to $\bP_{n,p,\bW}$ and $(E_n)_{n\in\N}$ be a sequence of independent $k_n$-dimensional random subspaces $E_n\subset\R^n$ distributed according to $\nu_{n,k_n}$. Assume that for each $n\in\N$, $X_n$ is independent of $E_n$. Then, if $k_n\to\infty$ and $\frac{k_n}{n}\to\lambda\in[0,1]$, as $n\to\infty$,
$$
\mathscr{X}_{n,p}:=n^{1/p}\sqrt{{\Gamma\big({1\over p}\big)\over p^{2/p}\Gamma\big({3\over p}\big)}}\,\|P_{E_n}X_n\|_2-\sqrt{k_n} \stackrel{\text{d}}{\longrightarrow} N\,,
$$
where $N$ is a centered Gaussian random variable with variance
\[
\sigma^2(p,\lambda):= {\lambda\over 4}{\Gamma({1\over p})\Gamma({5\over p})\over\Gamma({3\over p})^2}-\lambda\Big({3\over 4}-{1\over p}+{1 \over p^2}\Big)+{1\over 2}.
\]
\end{thm}

Moreover, if $k_n$ tends to infinity fast enough, then we obtain the following Berry-Esseen type bound measuring the speed of convergence in the previous central limit theorem.

\begin{thm}\label{thm:berry-esseen}
Under the assumptions of Theorem \ref{thm:CLTEuclideanNormProjections} and if additionally we assume $\frac{k_n}{n^{2/3}}\to\infty$, as $n\to\infty$, then there exists an absolute constant $\alpha\in(0,\infty)$ and a constant $C_p\in(0,\infty)$ only depending on $p$ such that, for any $n\geq 2$,
\begin{align*}\label{eq:BerryEsseen}
\sup_{t\in\R}\left|F_{n,p}(t)-\Phi(t)\right|&\leq C_p\max\left\{\frac{\log k_n}{\sqrt{k_n}},\frac{n}{k_n^{3/2}},\Big|\frac{k_n}{n}-\lambda\Big|\right\} \cr
&+\Pro\left(|W|>\frac{\alpha pn\log k_n}{k_n}\right)
+2\Pro\left(|W|>\sqrt{\frac{n^2\log k_n}{k_n}}\right),
\end{align*}
where $W$ is a random variable with distribution $\bW$, $F_{n,p}$ denotes the distribution function of $\mathscr{X}_{n,p}$ and $\Phi$ the distribution function of the Gaussian random variable $N$ from Theorem \ref{thm:CLTEuclideanNormProjections}.
\end{thm}

For the examples (i), (ii) and (iii) of distributions discussed above (before Theorem \ref{thm:CLTEuclideanNormProjections}) the probabilities in Theorem \ref{thm:berry-esseen} involving $W$ will either be $0$ or exponentially small and can thus be absorbed by the constant $C_p$. The upper bound for $\sup_{t\in\R}\left|F_{n,p}(t)-\Phi(t)\right|$ then reduces to the maximum term in Theorem \ref{thm:berry-esseen}.

Let us finally discuss the remaining case $p=\infty$. Here we only consider the uniform distribution on $\B_\infty^n=[-1,1]^n$ and obtain the following central limit theorem as well as a Berry-Esseen type rate of convergence when the subspace dimensions increase fast enough.

\begin{thm}\label{thm:CLTEuclideanNormProjections infinity}
Let $(k_n)_{n\in\N}$ be a sequence in $\N$ with $k_n\in\{1,\dots,n\}$, $(X_n)_{n\in\N}$ be a sequence of independent random vectors uniformly distributed in $\B_\infty^n$ and $(E_n)_{n\in\N}$ be a sequence of independent $k_n$-dimensional random subspaces $E_n\subset\R^n$ that are distributed according to $\nu_{n,k_n}$. Assume that for each $n\in\N$, $X_n$ is independent of $E_n$.
\begin{itemize}
\item[(a)] Then, if $k_n\to\infty$ and $\frac{k_n}{n}\to\lambda\in[0,1]$, as $n\to\infty$,
$$
\mathscr{X}_{n,\infty}:=\sqrt{3}\,\|P_{E_n}X_n\|_2-\sqrt{k_n} \stackrel{\text{d}}{\longrightarrow} N
$$
where $N$ is a centered Gaussian random variable with variance
\[
\sigma^2(\infty,\lambda):=\frac{1}{2}-{3\lambda\over 10}.
\]
\item[(b)] Moreover, if we assume $\frac{k_n}{n^{2/3}}\to\infty$, as $n\to\infty$, then there exists an absolute constant $C_\infty\in(0,\infty)$ such that, for any $n\geq 2$,
\[
\sup_{t\in\R}\left|F_{n,\infty}(t)-\Phi(t)\right|\leq C_\infty\max\left\{\frac{\log k_n}{\sqrt{k_n}},\frac{n}{k_n^{3/2}},\Big|\frac{k_n}{n}-\lambda\Big|\right\},
\]
where $F_{n,\infty}$ is the distribution function of $\mathscr X_{n,\infty}$ and $\Phi$ the one of the Gaussian random variable $N$ from part (a).
\end{itemize}
\end{thm}

\begin{rmk}
The asymptotic variance $\sigma^2(\infty,\lambda)$ in Theorem \ref{thm:CLTEuclideanNormProjections infinity} (a) appears as the limit of the asymptotic variances $\sigma^2(p,\lambda)$ from Theorem \ref{thm:CLTEuclideanNormProjections}, as $p\to\infty$. Indeed, we have
\begin{align*}
\sigma^2(p,\lambda) = {\lambda\over 4}{9\over 5}{\Gamma(1+{1\over p})\Gamma(1+{5\over p})\over\Gamma(1+{3\over p})^2}-\lambda\Big({3\over 4}+{1\over p}\Big)+{1\over 2} \,\,\stackrel{p\to\infty}{\longrightarrow}\,\, {9\lambda\over 20}-{3\lambda\over 4}+{1\over 2}={1\over 2}-{3\lambda\over 10} \,,
\end{align*}
as desired.
\end{rmk}

\begin{rmk}
The central limit theorem in Theorem \ref{thm:CLTEuclideanNormProjections} and Theorem \ref{thm:CLTEuclideanNormProjections infinity} (a) holds under the condition that $k_n\to\infty$. Against this light the additional condition that $k_n/n^{2/3}\to\infty$ in Theorem \ref{thm:berry-esseen} and Theorem \ref{thm:CLTEuclideanNormProjections infinity} (b) seems to be suboptimal and appears for technical reasons in our proof. To remove this condition is an open problem we leave for future research.
\end{rmk}

The rest of this paper is structured as follows. In Section \ref{sec:PrelimNotation} below we introduce our general notation as well as the probabilistic and geometric background material. Section \ref{sec:clt} is then devoted to the proof of the central limit theorem, where we consider separately the cases $1\leq p<\infty$ (Part A) and $p=\infty$ (Part B). The corresponding Berry-Esseen bounds on the rate of convergence in our central limit theorems are presented in Section \ref{sec:berry-esseen}. The last part, Section \ref{sec:ldps}, briefly discusses and sketches the extension of the large deviations results from \cite{APT16} to the class of probability measures on $\B_p^n$ considered in this work.

\section{Preliminaries and notation}\label{sec:PrelimNotation}

\subsection{Notation}

In this paper we will be working in $\R^n$ equipped with the standard Euclidean structure. We shall use the notation $|\,\cdot\,|$ to indicate the Lebesgue measure of the argument set, whose dimension will always be clear from the context. We will also write $|\cdot|$ to denote the modulus of a real or complex number, but the meaning will always be unambiguous.

For any $1\leq k\leq n$ we will denote by $\GG_{n,k}$ the Grassmannian of $k$-dimensional linear subspaces in $\R^n$ endowed with the unique Haar probability measure $\nu_{n,k}$, which is invariant under the action of the orthogonal group $O(n)$. By the uniqueness of the Haar measure it can be identified
with the image of the Haar probability measure $\widetilde{\nu}$ on $O(n)$ under the map $O(n)\to \GG_{n,k}$, $T\mapsto TE_0$, where $E_0:=\textrm{span}(\{e_1,\dots,e_k\})$ and $\{e_i\}_{i=1}^n$ is the canonical basis of $\R^n$.

We will use the Landau symbol $o(f)$, and may write $\psi\in o(f)$, to denote the class of functions $\psi:\R^n\to\R$ for which
\[
\lim_{x\to 0}\bigg|\frac{\psi(x)}{f(x)}\bigg|=0
\] or, equivalently, that for all $M\in(0,\infty)$ there exists $\delta>0$ such that for all $x\in\R^n$ with $\|x\|_2<\delta$,
\[
|\psi({\bf x})| \leq M |f({\bf x})|\,.
\]
We will write $\mathcal O(f)$ to represent the class of functions $\Psi:\R^n\to\R$ that satisfy that there exist $M,\delta>0$ such that, for any $\|{\bf x}\|_2<\delta$,
\[
|\Psi({\bf x})| \leq M |f({\bf x})|\,.
\]
We will indicate by $g_1(x)=g_2(x)+o(f(x))$ or $g_1(x)=g_2(x)+\mathcal{O}(f(x))$ the existence of a function $\psi\in o(f)$ or $\Psi\in\mathcal{O}(f)$ such that $g_1(x)=g_2(x)+\psi(x)$ or $g_1(x)=g_2(x)+\Psi(x)$, respectively.

\subsection{Definitions and results in probability theory}

Given a random variable $X$ on a probability space $(\Omega,{\mathcal A},\Pro)$, its distribution function is $F(t)=\Pro(X\leq t)$, $t\in\R$. Its expectation and variance with respect to $\Pro$ will be denoted by $\E X$ and $\Var X$, respectively. For any pair of random variables $X,Y$, we will write $X\stackrel{\text{d}}{=}Y$ when $X$ and $Y$ have the same distribution function. The characteristic function of $X$ is $\varphi_X(t):=\E \,e^{itX}$, $t\in\R$. From the definition we have that if $X$ and $Y$ are independent random variables, then
$\varphi_{X+Y}(t)=\varphi_X(t)\varphi_Y(t)$ and, for any $b,c\in\R$, we have $\varphi_{b+cX}(t)=e^{itb}\varphi_X(ct)$. If $N$ is a Gaussian random variable with mean $\mu\in\R$ and variance $\sigma^2>0$, its characteristic function is
$$
\varphi_N(t)=e^{it\mu-\frac{1}{2}\sigma^2t^2},\qquad t\in\R,
$$
the characteristic function of $N^2$ is
\begin{equation}\label{eq:CharacteristicChiSquared1}
\varphi_{N^2}(t)=\frac{1}{(1-2it)^{1/2}}
\end{equation}
and, therefore, the characteristic function of a $\chi^2$-random variable $\chi_k^2$ with $k\in\N$ degrees of freedom (recall that this means that $\chi_k^2\stackrel{\text{d}}{=}N_1^2+\ldots+N_k^2$, where $N_1,\ldots, N_k$ are independent copies of $N$) is
$$
\varphi_{\chi_k^2}(t)=\frac{1}{(1-2it)^{k/2}}.
$$
A sequence of random variables $X_n$ is said to converge to a random variable $X$ in distribution if the sequence of distribution functions of $X_n$ converges to the distribution function $F$ of $X$ for every point of continuity of $F$. In such a case we will write $X_n\stackrel{\text{d}}{\longrightarrow}X$.  By Levy's continuity theorem \cite[Theorem 5.3]{Kallenberg}, the sequence $X_n$ converges in distribution to $X$ if and only if $\varphi_{X_n}(t)$ converges to $\varphi_X(t)$ pointwise. The following lemma gives an estimate between the difference of the distribution functions of two random variables in terms of the difference between their characteristic functions.

\begin{lemma}[\cite{Feller2}, Chapter XVI.3, Lemma 2]\label{lem:SmoothingTheorem}
Assume that $F$ is the distribution function of a centered random variable $X$ with characteristic function $\varphi_X$ and $G$ is the distribution function of a random variable $Y$ with characteristic function $\varphi_Y$. Suppose that $G$ is continuously differentiable with $|G'(x)|\leq\beta<\infty$ for all $x\in\R$ and that $\varphi_Y$ is continuously differentiable with $\varphi_Y'(0)=0$. Then, for any $\Theta>0$,
$$
\sup_{t\in\R}|F(t)-G(t)|\leq \frac{1}{\pi} \int_{-\Theta}^\Theta\frac{|\varphi_X(t)-\varphi_Y(t)|}{|t|}\dif t+\frac{24\beta}{\pi\Theta}\,.
$$
\end{lemma}
\begin{rmk}
In the special case that $Y$ is a centered Gaussian random variable with variance $\sigma^2>0$, we may take $\beta=\frac{1}{\sqrt{2\pi\sigma^2}}$ in Lemma \ref{lem:SmoothingTheorem}.
\end{rmk}

A sequence of random variables $X_n$ is said to converge to a random variable $X$ in probability if for every $\varepsilon>0$, $\lim_{n\to\infty}\Pro(|X_n-X|>\varepsilon)=0$. We denote this by $X_n\stackrel{\text{P}}{\longrightarrow}X$. If a sequence of random variables converges to $X$ in probability, then it also converges in distribution. The following result of Slutsky gives the convergence in distribution of the sum and the product of two sequences of random variables provided that one of the sequences converges in probability to a constant.

\begin{lemma}[\cite{BassStochasticProcesses}, Proposition A.42 (b)]\label{thm:Slutsky}
Let $(X_n)_{n\in\N}$, $(Y_n)_{n\in\N}$, and $X$ be random variables and $c\in\R$ a constant such that, as $n\to\infty$,
$$
X_n\stackrel{\text{d}}{\longrightarrow}X\quad\textrm{and}\quad Y_n\stackrel{\text{P}}{\longrightarrow}c.
$$
Then, as $n\to\infty$,
$$
Y_nX_n\stackrel{\text{d}}{\longrightarrow}cX\quad\textrm{and}\quad X_n+Y_n\stackrel{\text{d}}{\longrightarrow}X+c.
$$
\end{lemma}

We will also make use of the classical Berry-Esseen theorem for sums of independent and identically distributed random variables.

\begin{lemma}[\cite{Feller2}, Chapter XVI.5, Theorem 1]\label{lem:Berry-Esseen}
Let $(X_n)_{n\in\N}$ be a sequence of independent copies of a centered random variable $X$ such that $\sigma^2:=\E X^2\in(0,\infty)$ and $\varrho:=\E|X|^3<\infty$. Further, let $F_n$ be the distribution function of $\frac{1}{\sqrt{\sigma^2n}}\sum\limits_{i=1}^nX_i$. Then there exists an absolute constant $C\in(0,\infty)$ such that
\[
\sup_{t\in\R} |F_n(t) - \Phi(t)| \leq C\frac{\varrho}{\sigma^3\sqrt{n}},
\]
where $\Phi$ is the distribution function of a standard Gaussian random variable.
\end{lemma}

\begin{rmk}
In this paper we shall work with the value $C=1$ for the constant in Lemma \ref{lem:Berry-Esseen}, although sharper estimates for $C$ are known.
\end{rmk}

\subsection{Geometry of $\ell_p^n$-balls}

Let $n\in\N$ and consider the $n$-dimensional space $\R^n$. For any $1\leq p\leq\infty$, the $\ell_p^n$-norm, $\|\cdot\|_p$, of a vector ${\bf x}=(x_1,\ldots,x_n)\in\R^n$ is given by
\[
\|{\bf x}\|_p := \begin{cases}
\Big(\sum\limits_{i=1}^n|x_i|^p\Big)^{1/p} &: p<\infty\\
\max\{|x_1|,\ldots,|x_n|\} &: p=\infty\,.
\end{cases}
\]
For any $n$ and $p$ we will denote by $\B_p^n:=\{{\bf x}\in\R^n:\|{\bf x}\|_p\leq 1\}$ the $\ell_p^n$-ball in $\R^n$ and by $\SSS_p^{n-1}:=\{{\bf x}\in\R^n:\|{\bf x}\|_p=1\}$ the corresponding unit sphere. The restriction of the Lebesgue measure to $\B_p^n$ provides a natural volume measure on $\B_p^n$. Although one could supply $\SSS_p^{n-1}$ with the $(n-1)$-dimensional Hausdorff measure, the so-called cone measure turns out to be more useful as explained later (see \cite{NaorTAMS} for the relation between these two measures). For a measurable set $A\subseteq\SSS_p^{n-1}$, the cone (probability) measure of $A$ is defined by
\[
\mu_p(A) := \frac{\big|\big\{r{\bf x}\,:\,{\bf x}\in A,\,r\in[0,1]\big\}\big|}{|\B_p^n|}\,.
\]
We remark that the cone measure $\mu_p$ coincides with the $(n-1)$-dimensional Hausdorff probability measure on $\SSS_p^{n-1}$ if and only if $p\in\{1,2,\infty\}$. In particular, this means that $\mu_2$ is the same as $\sigma_{n-1}$, the normalized spherical Lebesgue measure.

The proofs of our results rely on the following probabilistic representation from \cite[Theorem 3]{BartheGuedonEtAl} for the probability measures $\bP_{n,p,\bW}$ on $\B_p^n$ defined in \eqref{eq:DefMeasurePnpW}. Recall that $\bW$ can be any probability measure on $[0,\infty)$.

\begin{proposition}\label{prop:schechtman zinn}
Let $n\in\N$ and $1\leq p<\infty$. Suppose that $Z_1,\ldots,Z_n$ are independent $p$-generalized Gaussian random variables whose distribution has density
$$
f_p(x):= {1\over 2p^{1/p}\Gamma(1+{1\over p})}\,e^{-|x|^p/p}
$$
with respect to the Lebesgue measure on $\R$. Consider the random vector $Z:=(Z_1,\ldots,Z_n)\in\R^n$ and let $W$ be a random variable with distribution $\bW$, which is independent from $Z$. Then the $n$-dimensional random vector
$$
{Z\over(\|Z\|_p^p+W)^{1/p}}
$$
has distribution $\bP_{n,p,\bW}$.
\end{proposition}

In the rest of this paper and for $1\leq p<\infty$, $(Z_i)_{i\in\N}$ will denote a sequence of independent $p$-generalized Gaussian random variables with density $f_p$. When $p=\infty$ they will be understood as uniform random variables in $[-1,1]$.  All these random variables are assumed to be independent. Moreover, we assume that $W$ is a random variable with distribution $\bW$, which is independent of the sequence $(Z_i)_{i\in\N}$. Finally, $(g_i)_{i\in\N}$ will denote a sequence of independent standard Gaussian random variables that are independent of $(Z_i)_{i\in\N}$ and of $W$.

Using the previous representation of the measure $\bP_{n,p,\bW}$, the following representation for the Euclidean norm of a random projection of a random vector in $\B_p^n$ distributed according to $\bP_{n,p,\bW}$ can be proved along the lines of \cite[Theorem 3.1]{APT16} and for this reason we skip the details.

\begin{proposition}\label{thm:RepresentationAnnealed}
Let $\bW$ be a probability measure on $[0,\infty)$. For any $1\leq p\leq\infty$, $n\in\N$ and $k\in\{1,\ldots,n\}$ let $X$ be a random vector in $\B_p^n$ distributed according to $\bP_{n,p,\bW}$ if $p<\infty$ or according to the uniform probability measure if $p=\infty$, and $E\in \GG_{n,k}$ be a random subspace distributed according to $\nu_{n,k}$, independent from $X$. Then,
\[
\| P_E X\|_2\stackrel{\text{d}}{=}
\begin{cases}
\frac{\left(\sum\limits_{i=1}^n Z_i^2\right)^{1/2}}{\left(\sum\limits_{i=1}^n \vert Z_i\vert ^p+W \right)^{1/p}}\frac{\left( \sum\limits_{i=1}^k g_i^2\right)^{1/2}}{\left( \sum\limits_{i=1}^n g_i^2\right)^{1/2}} &: 1\leq p < \infty \\
&\\
\frac{\left(\sum\limits_{i=1}^nZ_i^2\right)^{1/2}\left( \sum\limits_{i=1}^k g_i^2\right)^{1/2}}{\left( \sum\limits_{i=1}^n g_i^2\right)^{1/2}}&: p=\infty.
\end{cases}
\]
\end{proposition}
\medskip

The difference to \cite[Theorem 3.1]{APT16} is that for $1\leq p<\infty$ the denominator contains the factor $\left(\sum\limits_{i=1}^n \vert Z_i\vert ^p+W \right)^{1/p}$ instead of $\left(\sum\limits_{i=1}^n \vert Z_i\vert ^p\right)^{1/p}$ and the whole expression needs to be multiplied with a factor $U^{1/n}$, where $U$ is uniformly distributed on $[0,1]$ and independent from $Z_1,\ldots,Z_n$ and $g_1,\ldots,g_n$.  The reason for this difference lies in the fact that here we use the probabilistic representation of Proposition \ref{prop:schechtman zinn} taken from \cite{BartheGuedonEtAl}, while in \cite{APT16} we were relying on the more classical representation of Schechtman and Zinn \cite{SchechtmanZinn}. The advantage of using the former representation lies in the fact that more general probability distributions on $\B_p^n$ can be treated this way simultaneously.

\subsection{Central limit theorem for convex bodies}

In this section we recall Klartag's central limit theorem for convex bodies (see \cite{KlartagCLT,KlartagCLT2}) in the form taken from \cite{MeckesM12} for bodies with sufficiently many symmetries. 

\begin{proposition}\label{prop:KlartagCLT}
Let $X$ be a random vector uniformly distributed in a centred convex body $K\subset\mathbb{R}^n$ having covariance matrix equal to the identity matrix. Assume that $K$ is symmetric with respect to all coordinate hyperplanes in $\R^n$. Suppose that $k\in\N$ is such that $k\leq n^\kappa$ for some $\kappa<1/7$. Then there exists a measurable subset $\mathcal{E}\subset\mathbb{G}_{n,k}$ and absolute constants $c_1,c_2,c_3\in(0,\infty)$ with $\nu_{n,k}(\mathcal{E})\geq 1-e^{-c_1n^{c_2}}$ such that
$$
\sup_{E\in\mathcal{E}}\sup_{A\subseteq E}\big|\Pro(P_EX\in A)-\Pro(N_E\in A)\big|\leq c_3n^{-\zeta},
$$
where $\zeta:={1-7\kappa\over 6}$, the second supremum runs over all measurable subsets $A\subseteq E$ and where $N_E$ denotes a standard Gaussian random vector in $E$.
\end{proposition}
We shall demonstrate later (see Remark \ref{rem:KlartagCLT})
how for small values of subspace dimensions $k_n$ the result of Theorem \ref{thm:CLTEuclideanNormProjections} can be deduced from Klartag's central limit theorem. Within the setting of $\ell_p^n$-balls, we will have to choose $$
K = {\B_p^n\over|\B_p^n|^{1/n}L_{\B_p^n}}\qquad\text{with}\qquad L^2_{\B_p^n}:={\Gamma({3\over p})\Gamma(1+{n\over p})\over\Gamma({1\over p})\Gamma(1+{n+2\over p})}|\B_p^n|^{-2/n},
$$
that is,
\begin{equation}\label{eq:IsotropicLpnBall}
K = \sqrt{{\Gamma({1\over p})\Gamma(1+{n+2\over p})}\over \Gamma({3\over p})\Gamma(1+{n\over p})}\;\B_p^n,
\end{equation} in order to ensure that the covariance matrix of the random vector $X$ equals the identity matrix.

\section{Proof of the central limit theorems}\label{sec:clt}

In this section we will give the proof of Theorem \ref{thm:CLTEuclideanNormProjections}. The proofs are presented separately for $p<\infty$ and $p=\infty$. The following lemma collects the value of some of the constants that frequently appear in our computations
and is proved by direct computations.

\begin{lemma}\label{lem:Constants}
Let $1\leq p \leq \infty$ and $Z_p$ be a $p$-generalized Gaussian random variable. Then,
\[
M_p(q):=\E |Z_p|^q= \frac{p^{q/p}}{q+1}\frac{\Gamma\left(1+\frac{q+1}{p}\right)}{\Gamma\left(1+1/p\right)}\quad\text{and}\quad M_\infty(q):=\E |Z_\infty|^q=\frac{1}{q+1}.
\]
Consequently, for any $q,r\geq1$, we have that
\[
\Var |Z_p|^q = M_p(2q)-M_p(q)^2
\]
and
\[
\Cov\big(|Z_p|^q,|Z_p|^r\big) = M_p(q+r)-M_p(q)M_p(r).
\]
In particular, if $p=\infty$,
\[
\Var |Z_\infty|^q =\frac{q^2}{(2q+1)(q+1)^2}
\]
and
\[
\Cov\big(|Z_\infty|^q,|Z_\infty|^r\big) = \frac{qr}{(q+r+1)(q+1)(r+1)}\,.
\]
\end{lemma}
\begin{proof}
First, let $p<\infty$. Recalling the definition of the density $f_p$ of $Z_p$, the result follows from
\begin{align*}
M_p(q) =\E |Z_p|^q = \int_{-\infty}^\infty |x|^q\,f_p(x)\,\dint x = {1\over p^{1/p}\Gamma\big(1+{1\over p}\big)}\int_0^\infty x^q\,e^{-x^p/p}\,\dint x
\end{align*}
and by a direct computation of the integral in terms of a gamma function. The covariance expression is a consequence of
\begin{align*}
\Cov\big(|Z_p|^q,|Z_p|^r\big) = \int_{-\infty}^\infty |x|^{q+r}\,f_p(x)\,\dint x - \Big(\int_{-\infty}^\infty |x|^q\,f_p(x)\,\dint x\Big)\Big(\int_{-\infty}^\infty |x|^r\,f_p(x)\,\dint x \Big).
\end{align*}
The case $p=\infty$ can be treated similarly by interpreting $f_\infty$ as the density of the uniform distribution on $[-1,1]$.
\end{proof}

\subsection{Part A - The case $1\leq p<\infty$}

We will now prove Theorem \ref{thm:CLTEuclideanNormProjections}.

\begin{proof}[Proof of Theorem \ref{thm:CLTEuclideanNormProjections}]

For $n\in\N$, we define the following random variables:
\[
\displaystyle{\xi_n^{(1)} := \frac{\sum\limits_{i=1}^n \big(|Z_i|^2-M_p(2)\big)}{\sqrt{n}}} \qquad\text{and}\qquad
\displaystyle{\xi^{(2)}_n := \frac{\sum\limits_{i=1}^n \big(|Z_i|^p-M_p(p)\big)}{\sqrt{n}}}
\]
as well as
\[
\xi^{(3)}_n := \frac{\sum\limits_{i=1}^{k_n}\big(|g_i|^2-M_2(2)\big)}{\sqrt{k_n}}\qquad\text{and}\qquad \xi^{(4)}_n := \frac{\sum\limits_{i=1}^{n}\big(|g_i|^2-M_2(2)\big)}{\sqrt{n}}\,.
\]
Since $M_2(2)=1$  and $M_p(p)=1$ by Lemma \ref{lem:Constants}, $\xi_n^{(2)}$, $\xi_n^{(3)}$ and $\xi_n^{(4)}$ reduce to
\[
\qquad \xi_n^{(2)} = \frac{\sum\limits_{i=1}^{n}\big(|Z_i|^p-1\big)}{\sqrt{n}},\quad\xi^{(3)}_n = \frac{\sum\limits_{i=1}^{k_n}\big(|g_i|^2-1\big)}{\sqrt{k_n}}\quad\text{and}\quad \xi_n^{(4)} = \frac{\sum\limits_{i=1}^{n}\big(|g_i|^2-1\big)}{\sqrt{n}}\,.
\]
Using the probabilistic representation of the Euclidean norm of the projection (Proposition \ref{thm:RepresentationAnnealed}) and rewriting the expressions that appear in terms of $\xi^{(1)}_n,\xi^{(2)}_n,\xi^{(3)}_n$ and $\xi^{(4)}_n$, we obtain
\begin{align}\label{eq:NormProjectionRepresentation}
\|P_{E_n}X_n\|_2 \overset{\text{d}}{=}{\sqrt{k_nM_p(2)}\over n^{1/p}}\,{\Big(1+{\xi_n^{(1)}\over\sqrt{n}\,M_p(2)}\Big)^{1/2}\Big(1+{\xi_n^{(3)}\over\sqrt{k_n}}\Big)^{1/2}\over\Big(1+{\xi_n^{(2)}\over\sqrt{n}}+{W\over n}\Big)^{1/p}\,\Big(1+{\xi_n^{(4)}\over\sqrt{n}}\Big)^{1/2}}\,.
\end{align}
Let us define the function
\[
F:\R^5\to \R,\qquad F(x_1,x_2,x_3,x_4,x_5) := \frac{\Big(1+\frac{x_1}{M_p(2)}\Big)^{1/2}}{\Big(1+x_2+x_5\Big)^{1/p}}\,\frac{\big(1+x_3\big)^{1/2}}{\big(1+x_4\big)^{1/2}}.
\]
Then,
\[
{n^{1/p}\over \sqrt{k_nM_p(2)}}\,\|P_{E_n}X_n\|_2  \stackrel{\text{d}}{=} F\bigg(\frac{\xi_n^{(1)}}{\sqrt{n}},\frac{\xi_n^{(2)}}{\sqrt{n}},\frac{\xi_n^{(3)}}{\sqrt{k_n}},\frac{\xi_n^{(4)}}{\sqrt{n}},\frac{W}{n}\bigg)\,.
\]
Let us point out that since $W$ is a positive random variable, $\frac{\xi_n^{(2)}}{\sqrt{n}}>-1$ with probability 1, and $\frac{\xi_n^{(4)}}{\sqrt{n}}>-1$ with probability 1, the random vector $\bigg(\frac{\xi_n^{(1)}}{\sqrt{n}},\frac{\xi_n^{(2)}}{\sqrt{n}},\frac{\xi_n^{(3)}}{\sqrt{k_n}},\frac{\xi_n^{(4)}}{\sqrt{n}},\frac{W}{n}\bigg)$ belongs to $D$, the domain of $F$, with probability 1. Using a Taylor expansion around the origin, we obtain that for every ${\bf x}=(x_1,x_2,x_3,x_4,x_5)\in D$,
\begin{equation}\label{eq:TaylorExpansionIdentity}
F({\bf x})
 = 1 + \frac{x_1}{2M_p(2)} - \frac{x_2}{p} + \frac{x_3}{2}-\frac{x_4}{2} - {x_5\over p} +\mathcal O\big(\|{\bf x}\|_2^2\big),
\end{equation}
where the Landau symbol stands for a function $\Psi_{p}:D\subseteq\R^5\to\R$ with the property that there are two constants $M,\delta>0$ such that $|\Psi_p({\bf x})|\leq M\,\Vert {\bf x}\Vert_2^2$ whenever $\Vert {\bf x}\Vert_2<\delta$. For this function $\Psi_p$, observing that
\[
\frac{1}{\sqrt{M_p(2)}}=\sqrt{\frac{\Gamma(\frac{1}{p})}{p^{2/p}\Gamma(\frac{3}{p})}},
\]
taking into account that the identity \eqref{eq:TaylorExpansionIdentity} holds for every $x\in D$, and defining $\lambda_n:=k_n/ n$, we can write
\begin{align*}
\mathscr X_{n,p} & \stackrel{\text{d}}{=} \sqrt{k_n}\,\left[\frac{\xi_n^{(1)}}{2M_p(2)\sqrt{n}} - \frac{\xi_n^{(2)}}{p\sqrt{n}} + \frac{\xi_n^{(3)}}{2\sqrt{k_n}}-\frac{\xi_n^{(4)}}{2\sqrt{n}} - \frac{W}{pn}\right. \\
& \left.\qquad\qquad\qquad\qquad +\Psi_p\left(\frac{\xi_n^{(1)}}{\sqrt{n}},\frac{\xi_n^{(2)}}{\sqrt{n}},\frac{\xi_n^{(3)}}{\sqrt{k_n}},\frac{\xi_n^{(4)}}{\sqrt{n}},\frac{W}{n}\right)\right] \\
&=\sqrt{\lambda_n}\,\frac{\xi_n^{(1)}}{2M_p(2)} - \sqrt{\lambda_n}\,\frac{\xi_n^{(2)}}{p} + \frac{\xi_n^{(3)}}{2}-\sqrt{\lambda_n}\,\frac{\xi_n^{(4)}}{2}- \sqrt{\lambda_n}\,\frac{W}{p\sqrt{n}} \\
&\qquad\qquad\qquad\qquad+ \sqrt{k_n}\, \Psi_p\left(\frac{\xi_n^{(1)}}{\sqrt{n}},\frac{\xi_n^{(2)}}{\sqrt{n}},\frac{\xi_n^{(3)}}{\sqrt{k_n}},\frac{\xi_n^{(4)}}{\sqrt{n}},\frac{W}{n}\right).
\end{align*}
For each $n\in\N$ let us define the following three random variables:
\[
Y_n^{(1)}:=\frac{\sqrt{\lambda_n}\xi_n^{(1)}}{2M_p(2)} - \frac{\sqrt{\lambda_n}\xi_n^{(2)}}{p} + \frac{\xi_n^{(3)}}{2}-\frac{\sqrt{\lambda_n}\xi_n^{(4)}}{2}\,,\qquad Y_n^{(2)}:=-\frac{\sqrt{\lambda_n}W}{p\sqrt n}
\]
and
\[
Y_n^{(3)}:=\sqrt{k_n}\,\Psi_p\,\bigg(\frac{\xi_n^{(1)}}{\sqrt{n}},\frac{\xi_n^{(2)}}{\sqrt{n}},\frac{\xi_n^{(3)}}{\sqrt{k_n}},\frac{\xi_n^{(4)}}{\sqrt{n}},\frac{W}{n}\bigg).
\]
We will now show that the first random variable, $Y_n^{(1)}$, converges to a centered Gaussian with the variance as stated in the theorem, while the other two, $Y_n^{(2)}$ and $Y_n^{(3)}$, converge in probability to $0$, as $n\to\infty$. We do this in three steps.
\vskip 1mm
\textbf{Step 1.}
For $i=1,\ldots,n$ define the random variables $\mathcal{Z}_i$ and $\mathcal{G}_i$ by
\begin{align}\label{eq:DefZiGi}
\mathcal{Z}_i := \frac{|Z_i|^2-M_p(2)}{2M_p(2)}-\frac{|Z_i|^p-1}{p}\qquad\text{and}\qquad\mathcal{G}_i := |g_i|^2-1,
\end{align}
and consider the three normalized sums
\begin{align*}
\sqrt{\frac{\lambda_n}{n}}\sum_{i=1}^n \mathcal Z_i\,,\qquad\frac{1-\lambda_n}{2\sqrt{\lambda_n}\sqrt{n}}\sum_{i=1}^{k_n} \mathcal G_i\,,\qquad\text{and}\qquad-\frac{\sqrt{\lambda_n}}{2\sqrt{n}}\sum_{i=k_n+1}^{n}\mathcal G_i\,.
\end{align*}
These three sums are mutually independent and each one of them is a sum of independent and identically distributed random variables. Moreover, we observe that
\begin{align*}
 &\sqrt{\frac{\lambda_n}{n}}\sum_{i=1}^n \mathcal Z_i+\frac{1-\lambda_n}{2\sqrt{\lambda_n}\sqrt{n}}\sum_{i=1}^{k_n} \mathcal G_i-\frac{\sqrt{\lambda_n}}{2\sqrt{n}}\sum_{i=k_n+1}^{n}\mathcal G_i\\
&=\frac{\sqrt{\lambda_n}\xi_n^{(1)}}{2M_p(2)} - \frac{\sqrt{\lambda_n}\xi_n^{(2)}}{p} + \frac{\xi_n^{(3)}}{2}-\frac{\sqrt{\lambda_n}\xi_n^{(4)}}{2}=Y_n^{(1)}\,.
\end{align*}
For any $t\in\R$, the characteristic function of $Y_n^{(1)}$ is therefore
\begin{align*}
\varphi_{Y_n^{(1)}}(t) & =\varphi_{\mathcal Z_1}^n\left(\frac{t\sqrt{\lambda_n}}{\sqrt{n}}\right)\varphi^{\lambda_nn}_{\mathcal G_1}\left(\frac{t(1-\lambda_n)}{2\sqrt{\lambda_n}\sqrt{n}}\right)\varphi^{(1-\lambda_n)n}_{\mathcal G_1}\left(\frac{-t\sqrt{\lambda_n}}{2\sqrt{n}}\right).
\end{align*}
Using a Taylor expansion of the exponential function at $0$ and the Bienaym\'e identity, we observe that
\begin{align*}
&\varphi_{\mathcal Z_1}\left(\frac{t\sqrt{\lambda_n}}{\sqrt{n}}\right)\\
& = 1-\Var \mathcal Z_1\,\frac{\lambda_nt^2}{2n}+ o\left(\frac{t^2\lambda_n}{n}\right)\cr
&=1-\left(\frac{\Var|Z_1|^2}{4M_p^2(2)}+\frac{\Var|Z_1|^p}{p^2}-\frac{\Cov(|Z_1|^2,|Z_1|^p)}{pM_p(2)}\right)\frac{\lambda_nt^2}{2n}+o\left(\frac{t^2\lambda_n}{n}\right),
\end{align*}
and similarly
\begin{align}
\label{eq:NoTaylor}\varphi_{\mathcal G_1}\left(\frac{t(1-\lambda_n)}{2\sqrt{\lambda_n}\sqrt{n}}\right) &=1-{\Var|g_1|^2\over 4}\,\frac{t^2(1-\lambda_n)^2}{2n\lambda_n}+o\left(\frac{t^2(1-\lambda_n)^2}{4\lambda_nn}\right),\\
\nonumber\varphi_{\mathcal G_1}\left(-\frac{t\sqrt{\lambda_n}}{2\sqrt{n}}\right)&=1-{\Var|g_1|^2\over 4}\,\frac{\lambda_nt^2}{2n}+o\left(\frac{t^2\lambda_n}{4n}\right).
\end{align}
Therefore, since by assumption $\lambda_nn=k_n\to\infty$, we obtain, for every $t\in\R$,
\[
\lim_{n\to\infty}\varphi_{Y^{(1)}_n}(t) = e^{-{t^2\over 2}\,\eta^2}\,,
\]
with the exponent $\eta^2\in\R$ given by
\begin{align*}
\eta^2 &:= \lambda \Var\mathcal Z_1 + \frac{(1-\lambda)\Var|g_1|^2}{4} \cr
&=\frac{\lambda\Var|Z_1|^2}{4M_p^2(2)}+\frac{\lambda\Var|Z_1|^p}{p^2}-\frac{\lambda\Cov(|Z_1|^2,|Z_1|^p)}{pM_p(2)}+\frac{(1-\lambda)\Var|g_1|^2}{4}\,.
\end{align*}
Thus, by Levy's continuity theorem, as $n\to\infty$, the random variable $Y_n^{(1)}$
converges in distribution to a centered Gaussian random variable with variance
\begin{align*}
\eta^2=\frac{\lambda M_p(4)}{4 M_p^2(2)}+\frac{\lambda M_p(2p)}{p^2}-\frac{\lambda M_p(p+2)}{pM_p(2)}-\lambda\Big[\frac{3}{4}-\frac{1}{p}+\frac{1}{p^2}\Big]+\frac{1}{2}\,.
\end{align*}
Using the explicit expressions in terms of gamma functions provided by Lemma \ref{lem:Constants}, we see that $\eta^2$ coincides with the constant $\sigma^2(p,\lambda)$ in the statement of the theorem.

\textbf{Step 2.} Since the random variables $\frac{W}{\sqrt n}$ converges to $0$ in probability as $n\to\infty$, also $Y_n^{(2)}=-\frac{\sqrt{\lambda_n}W}{p\sqrt n}$ converges to $0$ in probability.

\textbf{Step 3.} By Slutsky's theorem (see Lemma \ref{thm:Slutsky}) it is now left to prove that, as $n\to\infty$,
\[
Y_n^{(3)}=\sqrt{k_n}\,\Psi_p\,\bigg(\frac{\xi_n^{(1)}}{\sqrt{n}},\frac{\xi_n^{(2)}}{\sqrt{n}},\frac{\xi_n^{(3)}}{\sqrt{k_n}},\frac{\xi_n^{(4)}}{\sqrt{n}},\frac{W}{n}\bigg)
\]
converges to $0$ in probability as well. We observe that
\begin{align*}
Y_n^{(3)}
&=\sqrt{k_n}\,\bigg(\frac{(\xi_n^{(1)})^2}{n}+\frac{(\xi_n^{(2)})^2}{n}+\frac{(\xi_n^{(3)})^2}{k_n}+\frac{(\xi_n^{(4)})^2}{n}+\frac{W^2}{n^2}\bigg)\\
&\qquad\qquad\times\frac{
\Psi_p\Big(\frac{\xi_n^{(1)}}{\sqrt{n}},\frac{\xi_n^{(2)}}{\sqrt{n}},\frac{\xi_n^{(3)}}{\sqrt{k_n}},\frac{\xi_n^{(4)}}{\sqrt{n}},\frac{W}{n}\Big)}{\Big\|\Big(\frac{\xi_n^{(1)}}{\sqrt{n}},\frac{\xi_n^{(2)}}{\sqrt{n}},\frac{\xi_n^{(3)}}{\sqrt{k_n}},\frac{\xi_n^{(4)}}{\sqrt{n}},\frac{W}{n}\Big)\Big\|_2^2}\cr
&=\bigg(\sqrt{\lambda_n}\xi_n^{(1)}\frac{\xi_n^{(1)}}{\sqrt{n}}+\sqrt{\lambda_n}\xi_n^{(2)}\frac{\xi_n^{(2)}}{\sqrt{n}}+\xi_n^{(3)}\frac{\xi_n^{(3)}}{\sqrt{k_n}}+\sqrt{\lambda_n}\xi_n^{(4)}\frac{\xi_n^{(4)}}{\sqrt{n}}+\sqrt{\lambda_n}\frac{W^2}{n^{3/2}}\bigg) \\
&\qquad\qquad\times\frac{\Psi_p\Big(\frac{\xi_n^{(1)}}{\sqrt{n}},\frac{\xi_n^{(2)}}{\sqrt{n}},\frac{\xi_n^{(3)}}{\sqrt{k_n}},\frac{\xi_n^{(4)}}{\sqrt{n}},\frac{W}{n}\Big)}{\Big\|\Big(\frac{\xi_n^{(1)}}{\sqrt{n}},\frac{\xi_n^{(2)}}{\sqrt{n}},\frac{\xi_n^{(3)}}{\sqrt{k_n}},\frac{\xi_n^{(4)}}{\sqrt{n}},\frac{W}{n}\Big)\Big\|_2^2}\,.
\end{align*}
Also, we have
\begin{align*}
&\Pro\left(\frac{\Big|\Psi_p\Big(\frac{\xi_n^{(1)}}{\sqrt{n}},\frac{\xi_n^{(2)}}{\sqrt{n}},\frac{\xi_n^{(3)}}{\sqrt{k_n}},\frac{\xi_n^{(4)}}{\sqrt{n}},\frac{W}{n}\Big)\Big|}{\Big\|\Big(\frac{\xi_n^{(1)}}{\sqrt{n}},\frac{\xi_n^{(2)}}{\sqrt{n}},\frac{\xi_n^{(3)}}{\sqrt{k_n}},\frac{\xi_n^{(4)}}{\sqrt{n}},\frac{W}{n}\Big)\Big\|_2^2}>M\right)\cr
&\leq\Pro\left(\Big\|\Big(\frac{\xi_n^{(1)}}{\sqrt{n}},\frac{\xi_n^{(2)}}{\sqrt{n}},\frac{\xi_n^{(3)}}{\sqrt{k_n}},\frac{\xi_n^{(4)}}{\sqrt{n}},\frac{W}{n}\Big)\Big\|_2>\delta\right)\cr
&\leq\Pro\left(\frac{\xi_n^{(1)}}{\sqrt{n}}\geq\frac{\delta}{\sqrt 5}\right)+\Pro\left(\frac{\xi_n^{(2)}}{\sqrt{n}}\geq\frac{\delta}{\sqrt 5}\right)+\Pro\left(\frac{\xi_n^{(3)}}{\sqrt{k_n}}\geq\frac{\delta}{\sqrt 5}\right)\cr
& \qquad\qquad+\,\Pro\left(\frac{\xi_n^{(4)}}{\sqrt{n}}\geq\frac{\delta}{\sqrt 5}\right)+\Pro\left(\frac{W}{n}\geq\frac{\delta}{\sqrt 5}\right).
\end{align*}
Since by the weak law of large numbers all the random variables ${\xi_n^{(1)}\over\sqrt{n}},{\xi_n^{(2)}\over\sqrt{n}},{\xi_n^{(3)}\over\sqrt{k_n}},{\xi_n^{(4)}\over\sqrt{n}}$ and $W\over n$ converge to $0$ in probability, we have that these probabilities converges to $0$, as $n\to\infty$. Therefore, for any $\varepsilon>0$, we have that
\begin{align*}
&\Pro\left(\Big(\sqrt{\lambda_n}\xi_n^{(1)}\frac{\xi_n^{(1)}}{\sqrt{n}}+\sqrt{\lambda_n}\xi_n^{(2)}\frac{\xi_n^{(2)}}{\sqrt{n}}+\xi_n^{(3)}\frac{\xi_n^{(3)}}{\sqrt{k_n}}+\sqrt{\lambda_n}\xi_n^{(4)}\frac{\xi_n^{(4)}}{\sqrt{n}}+\sqrt{\lambda_n}\frac{W^2}{n^{3/2}}\Big)\right.\cr
& \qquad\qquad\qquad\left.\times\frac{\Big|\Psi_p\Big(\frac{\xi_n^{(1)}}{\sqrt{n}},\frac{\xi_n^{(2)}}{\sqrt{n}},\frac{\xi_n^{(3)}}{\sqrt{k_n}},\frac{\xi_n^{(4)}}{\sqrt{n}},\frac{W}{n}\Big)\Big|}{\Big\|\Big(\frac{\xi_n^{(1)}}{\sqrt{n}},\frac{\xi_n^{(2)}}{\sqrt{n}},\frac{\xi_n^{(3)}}{\sqrt{k_n}},\frac{\xi_n^{(4)}}{\sqrt{n}},\frac{W}{n}\Big)\Big\|_2^2}>\varepsilon\right)\cr
&\leq\Pro\left(\sqrt{\lambda_n}\xi_n^{(1)}\frac{\xi_n^{(1)}}{\sqrt{n}}+\sqrt{\lambda_n}\xi_n^{(2)}\frac{\xi_n^{(2)}}{\sqrt{n}}+\xi_n^{(3)}\frac{\xi_n^{(3)}}{\sqrt{k_n}}+\sqrt{\lambda_n}\xi_n^{(4)}\frac{\xi_n^{(4)}}{\sqrt{n}}+\sqrt{\lambda_n}\frac{W^2}{n^{3/2}}>\frac{\varepsilon}{M}\right)\cr
&\qquad\qquad\qquad+\Pro\left(\frac{\Big|\Psi_p\Big(\frac{\xi_n^{(1)}}{\sqrt{n}},\frac{\xi_n^{(2)}}{\sqrt{n}},\frac{\xi_n^{(3)}}{\sqrt{k_n}},\frac{\xi_n^{(4)}}{\sqrt{n}},\frac{W}{n}\Big)\Big|}{\Big\|\Big(\frac{\xi_n^{(1)}}{\sqrt{n}},\frac{\xi_n^{(2)}}{\sqrt{n}},\frac{\xi_n^{(3)}}{\sqrt{k_n}},\frac{\xi_n^{(4)}}{\sqrt{n}},\frac{W}{n}\Big)\Big\|_2^2}>M\right).
\end{align*}
Again, recall that by the weak law of large numbers, as $n\to\infty$, the random variables $\frac{\xi_n^{(1)}}{\sqrt{n}}$, $\frac{\xi_n^{(2)}}{\sqrt{n}}$, $\frac{\xi_n^{(3)}}{\sqrt{k_n}}$, $\frac{\xi_n^{(4)}}{\sqrt{n}}$ converge in probability to $0$ and observe that also $\frac{W^2}{n^{3/2}}$ converges in probability to $0$. Therefore, since by the classical central limit theorem for sums of independent random variables (see \cite[Proposition 5.9]{Kallenberg}), $\xi_n^{(1)}$, $\xi_n^{(2)}$, $\xi_n^{(3)}$, and $\xi_n^{(4)}$ converge in distribution to (non-independent) Gaussian random variables, as $n\to\infty$, by Slutsky's theorem (Lemma \ref{thm:Slutsky}) the random variable
\[
\sqrt{\lambda_n}\xi_n^{(1)}\frac{\xi_n^{(1)}}{\sqrt{n}}+\sqrt{\lambda_n}\xi_n^{(2)}\frac{\xi_n^{(2)}}{\sqrt{n}}+\xi_n^{(3)}\frac{\xi_n^{(3)}}{\sqrt{k_n}}+\sqrt{\lambda_n}\xi_n^{(4)}\frac{\xi_n^{(4)}}{\sqrt{n}}+\sqrt{\lambda_n}\frac{W^2}{n^{3/2}}
\]
converges to $0$ in probability, as $n\to\infty$. Consequently, the random variables $Y_n^{(3)}$
also converges to $0$ in probability, as $n\to\infty$.

The proof of Theorem \ref{thm:CLTEuclideanNormProjections} is now a direct consequence of Steps $1$ to $3$.
\end{proof}

\begin{rmk}\label{rem:KlartagCLT}
Let us notice that for small values of $k_n$ the result of our Theorem \ref{thm:CLTEuclideanNormProjections} can be deduced from Klartag's central limit theorem in Proposition \ref{prop:KlartagCLT}. Indeed, notice that if $k_n\leq n^\kappa$ for some $\kappa<1/7$ the Gaussian random variable $N$ in Theorem \ref{thm:CLTEuclideanNormProjections} has variance equal to $1/2$. According to Proposition \ref{prop:KlartagCLT} for every $k_n$ there exists $\mathcal{E}_n\subset\mathbb{G}_{n,k_n}$  with $\nu_{n,k_n}(\mathcal{E}_n)\geq 1-e^{-c_1n^{c_2}}$ such that
\begin{align}\label{eq:RemKlartag1}
\sup_{E\in\mathcal{E}_n}\sup_{A\subseteq E}\big|\Pro(P_{E}U_n\in A)-\Pro(N_{E}\in A)\big|\leq c_3n^{-\zeta},
\end{align}
where $U_n$ is uniformly distributed in the normalized $\ell_p^n$-ball defined in \eqref{eq:IsotropicLpnBall}. Letting $A_{t,p,n}^{E}$ be the ball in $E$ centred at the origin with radius
$$
(t+\sqrt{k_n})n^{-1/p}\sqrt{M_p(2)}\,L_{\B_p^n}^{-1}\,|\B_p^n|^{-1/n} =
 (t+\sqrt{k_n})\Big({p\over n}\Big)^{1/p}\sqrt{\Gamma(1+{n+2\over p})\over\Gamma(1+{n\over p})},
$$
we conclude from \eqref{eq:RemKlartag1}, taking into account that  the probability $\Pro(N_{E}\in A_{t,p,n}^{E})$ only depends on the dimension of $E$, that, for each $t\in\R$,
\begin{align}\label{eq:SplitTermsKlartag}
&\big|\Pro(\mathscr{X}_{n,p}\leq t)-\Pro(N\leq t)\big| \cr
&\leq \int_{\mathcal{E}_n}\big|\Pro(P_{E}U_n\in A_{t,p,n}^{E})-\Pro(N_{E}\in A_{t,p,n}^{E})\big|\,\nu_{n,k_n}(\dint E)\cr
&\qquad\qquad\qquad\qquad+\big|\Pro(N_{\R^{k_n}}\in A_{t,p,n}^{\R^{k_n}})-\Pro(N\leq t)\big|+2e^{-c_1n^{c_2}}\cr
&\leq c_4n^{-\zeta}+\big|\Pro(N_{\R^{k_n}}\in A_{t,p,n}^{\R^{k_n}})-\Pro(N\leq t)\big|
\end{align}
for some absolute constant $c_4\in(0,\infty)$. 

Notice now that
$$
\Pro(N_{\R^{k_n}}\in A_{t,p,n}^{\R^{k_n}})=\Pro\left(\sum_{i=1}^{k_n}\frac{g_i^2-1}{\sqrt{2k_n}}\leq\frac{(t+\sqrt{k_n})^2}{\sqrt{2k_n}}\Big({p\over n}\Big)^{2/p}{\Gamma(1+{n+2\over p})\over\Gamma(1+{n\over p})}-\sqrt{k_n\over 2}\right),
$$
where $g_1,\dots,g_{k_n}$ are independent standard Gaussian random variables. This yields that the second term in \eqref{eq:SplitTermsKlartag} is bounded above by
\begin{align}\label{eq:SecondSplitTermKlartag}
&\left|\Pro\left(\sum_{i=1}^{k_n}\frac{g_i^2-1}{\sqrt{2k_n}}\leq\frac{(t+\sqrt{k_n})^2}{\sqrt{2k_n}}\Big({p\over n}\Big)^{2/p}{\Gamma(1+{n+2\over p})\over\Gamma(1+{n\over p})}-\sqrt{k_n\over 2}\right)-\Pro\left(\sum_{i=1}^{k_n}\frac{g_i^2-1}{\sqrt{2k_n}}\leq\sqrt{2}t\right)\right|\cr
&\qquad\qquad\qquad+\left|\Pro\left(\sum_{i=1}^{k_n}\frac{g_i^2-1}{\sqrt{2k_n}}\leq\sqrt{2}t\right)-\Pro(N\leq t)\right|,
\end{align}
 By the classical central limit theorem for sums of independent random variables, if $k_n$ tends to $\infty$ the random variable $S_n:=\sum\limits_{i=1}^{k_n}\frac{g_i^2-1}{\sqrt{2k_n}}$ converges in distribution to a standard Gaussian random variable and, since $N$ is a Gaussian random variable with variance $\frac{1}{2}$, the second term in \eqref{eq:SecondSplitTermKlartag} tends to $0$, as $n\to\infty$. 

By the mean value theorem, calling $c_n:=\big({p\over n}\big)^{2/p}{\Gamma(1+{n+2\over p})\over\Gamma(1+{n\over p})}$, the first term in \eqref{eq:SecondSplitTermKlartag} is bounded by 
$$
\left|\frac{t^2c_n}{\sqrt{2k_n}}+\sqrt{2}t(c_n-1)+\sqrt{\frac{k_n}{2}}(c_n-1)\right|\sup_{x\in\R}f_n(x),
$$
where $f_n$ is the density of the random variable $S_n$. By Stirling's formula $c_n=1+\mathcal O\big(\frac{1}{n}\big)$. Moreover, the density version of the central limit theorem (see \cite[Theorem 3.1]{DensityCLT}) implies that the supremum is bounded. Therefore, the latter quantity tends to $0$, as $n\to\infty$.

Summarizing, we conclude that for $k_n$ with $k_n\to\infty$ and $k_n\leq n^\kappa$ our Theorem \ref{thm:CLTEuclideanNormProjections} (and Theorem \ref{thm:CLTEuclideanNormProjections infinity} (a)) is a consequence of the central limit theorems in \cite{KlartagCLT2,MeckesM12}.
\end{rmk}

\begin{rmk}
In the case that $k_n\to k\in\N$, as $n\to\infty$, we cannot use the Taylor approximation argument exploited in the previous proof, as the argument in the $o$-term in \eqref{eq:NoTaylor} does not tend to zero. However, it follows from \eqref{eq:NormProjectionRepresentation} that in this set-up
$$
\frac{n^{1/p}}{\sqrt{M_p(2)}}\|P_{E_n}X_n\|_2\overset{d}{=}\,{\Big(1+{\xi_n^{(1)}\over\sqrt{n}\,M_p(2)}\Big)^{1/2}\over\Big(1+{\xi_n^{(2)}\over\sqrt{n}}+{W\over n}\Big)^{1/p}\,\Big(1+{\xi_n^{(4)}\over\sqrt{n}}\Big)^{1/2}}\sqrt{\sum_{i=1}^{k_n}|g_i|^2}\,.
$$
By the weak law of large numbers the random variables $\xi_n^{(1)}\over\sqrt{n}$, $\xi_n^{(2)}\over\sqrt{n}$, and also $\xi_n^{(4)}\over\sqrt{n}$ and $W\over n$ converges to 0 in probability. Thus, using Lemma \ref{thm:Slutsky}, and the fact that convergence in distribution to a constant implies convergence in probability, we obtain that
$$
{\Big(1+{\xi_n^{(1)}\over\sqrt{n}\,M_p(2)}\Big)^{1/2}\over\Big(1+{\xi_n^{(2)}\over\sqrt{n}}+{W\over n}\Big)^{1/p}\,\Big(1+{\xi_n^{(4)}\over\sqrt{n}}\Big)^{1/2}}\stackrel{\text{P}}{\longrightarrow}1.
$$
By Levy's continuity theorem,
$$
\sum_{i=1}^{k_n}|g_i|^2\overset{d}{\longrightarrow}\chi_k^2,
$$
where $\chi_k^2$ is a $\chi^2$-random variable with $k$ degrees of freedom. Thus,
$$
\sqrt{\sum_{i=1}^{k_n}|g_i|^2}\overset{d}{\longrightarrow}\sqrt{\chi_k^2}.
$$
Using again Lemma \ref{thm:Slutsky} we obtain that
$$
\mathscr{X}_{n,p}=\frac{n^{1/p}}{\sqrt{M_p(2)}}\|P_{E_n}X_n\|_2-\sqrt{k_n} \overset{d}{\longrightarrow}\sqrt{\chi_k^2}-\sqrt{k},
$$
as $n\to\infty$. We remark that in the special case that $\bW$ is the exponential distribution with mean $1$ this can also be concluded from Klartag's central limit theorem from \cite{KlartagCLT,KlartagCLT2}. Note that this also holds in the case $p=\infty$ if $1/p$ is interpreted as $0$.
\end{rmk}

\subsection{Part B - The case $p=\infty$}

We will now present the proof of part (a) of Theorem \ref{thm:CLTEuclideanNormProjections infinity}.

\begin{proof}[Proof of Theorem \ref{thm:CLTEuclideanNormProjections infinity} (a)] The proof for $p=\infty$ is similar to that for $p<\infty$, but some technical details are different. For the sake of completeness we present it below.

For each $n\in\N$ let us define the random variables
\begin{align*}
{\xi^{(1)}_n := \frac{\sum\limits_{i=1}^n (|Z_i|^2-M_\infty(2))}{\sqrt{n}}}\,,\quad {\xi^{(2)}_n := \frac{\sum\limits_{i=1}^{k_n}(|g_i|^2-M_2(2))}{\sqrt{k_n}}=\frac{\sum\limits_{i=1}^{k_n}(|g_i|^2-1)}{\sqrt{k_n}}}
\end{align*}
and
$$
{\xi^{(3)}_n := \frac{\sum\limits_{i=1}^{n}(|g_i|^2-M_2(2))}{\sqrt{n}}=\frac{\sum\limits_{i=1}^{n}(|g_i|^2-1)}{\sqrt{n}}}.
$$
By Proposition \ref{thm:RepresentationAnnealed}, we have
\begin{align*}
\frac{\|P_{E_n}X_n\|_2}{\sqrt{k_nM_\infty(2)}}&\overset{\text{d}}{=}\frac{\Big(1+\frac{\xi_n^{(1)}}{\sqrt{n}M_\infty(2)}\Big)^{1/2}\Big(1+\frac{\xi_n^{(2)}}{\sqrt{k_n}}\Big)^{1/2}}{\Big(1+\frac{\xi_n^{(3)}}{\sqrt{n}}\Big)^{1/2}} 
\overset{\text{d}}{=}F\bigg(\frac{\xi_n^{(1)}}{\sqrt{n}},\frac{\xi_n^{(2)}}{\sqrt{k_n}},\frac{\xi_n^{(3)}}{\sqrt{n}}\bigg),
\end{align*}
where
\[
F(x_1,x_2,x_3) := \frac{\left(1+\frac{x_1}{M_\infty(2)}\right)^{1/2}(1+x_2)^{1/2}}{(1+x_3)^{1/2}}.
\]
 Using Taylor expansion at the origin, we obtain that for ${\bf x}=(x_1,x_2,x_3)$,
\begin{align*}
F({\bf x})
& = 1 + \frac{x_1}{2M_\infty(2)} + \frac{x_2}{2}-\frac{x_3}{2} + {\mathcal O}\big(\|{\bf x}\|_2^2\big).
\end{align*}
Again, the Landau symbol stands for a function $\Psi_\infty:\R^3\to\R$ for which there are two constants $M,\delta>0$ such that $|\Psi_\infty({\bf x})|\leq M\Vert {\bf x}\Vert_2^2$ for any $\Vert {\bf x}\Vert_2<\delta$. As a consequence, we obtain
\begin{align*}
\mathscr X_{n,\infty}& \stackrel{\text{d}}{=}\frac{\sqrt{\lambda_n}\xi_n^{(1)}}{2M_\infty(2)} + \frac{\xi_n^{(2)}}{2}-\frac{\sqrt{\lambda_n}\xi_n^{(3)}}{2}+ \sqrt{k_n} \Psi_\infty\left(\frac{\xi_n^{(1)}}{\sqrt{n}},\frac{\xi_n^{(2)}}{\sqrt{k_n}},\frac{\xi_n^{(3)}}{\sqrt{n}}\right). 
\end{align*}
Next, we call, for $i=1,\ldots,n$,
\begin{align*}
\mathcal{Z}_i:=|Z_i|^2-M_\infty(2)\qquad\text{and}\qquad\mathcal{G}_i:=|g_i|^2-1
\end{align*}
and, for each $n\in\N$,
\begin{align*}
Y_n^{(1)} := \frac{\sqrt{\lambda_n}}{2\sqrt{n}M_\infty(2)}\sum_{i=1}^n \mathcal{Z}_i+\frac{1-\lambda_n}{2\sqrt{k_n}}\sum_{i=1}^{k_n}\mathcal{G}_i-\frac{\sqrt{\lambda_n}}{2\sqrt{n}}\sum_{i=k_n+1}^{n}\mathcal{G}_i.
\end{align*}
These three sums are mutually independent and each one of them is a sum of independent identically distributed random variables. Moreover, we observe that
$$
Y^{(1)}_n=\frac{\sqrt{\lambda_n}\xi_n^{(1)}}{2M_\infty(2)} + \frac{\xi_n^{(2)}}{2}-\frac{\sqrt{\lambda_n}\xi_n^{(3)}}{2}.
$$
Therefore, for any $t\in\R$, the characteristic function of $Y^{(1)}_n$ is
$$
\varphi_{Y^{(1)}_n}(t)=\varphi_{\mathcal{Z}_1}^n\left(\frac{t\sqrt{\lambda_n}}{2\sqrt{n}M_\infty(2)}\right)\varphi^{\lambda_nn}_{\mathcal{G}_1}\left(\frac{t(1-\lambda_n)}{2\sqrt{\lambda_n}\sqrt{n}}\right)\varphi^{(1-\lambda_n)n}_{\mathcal{G}_1}\left(\frac{-t\sqrt{\lambda_n}}{2\sqrt{n}}\right).
$$
Since
\begin{align*}
\varphi_{\mathcal{Z}_1}\left(\frac{t\sqrt{\lambda_n}}{2M_\infty(2)\sqrt{n}}\right)
 &= 1-\frac{\Var|Z_1|^2}{4M_\infty^2(2)}\frac{t^2\lambda_n}{2n}+o\left(\frac{t^2\lambda_n}{4nM_\infty(2)^2}\right),\\
 \varphi_{\mathcal{G}_1}\left(\frac{t(1-\lambda_n)}{2\sqrt{\lambda_n}\sqrt{n}}\right) &= 1-\frac{\Var|g_1|^2(1-\lambda_n)^2}{4\lambda_n}\frac{t^2}{2n}+o\left(\frac{t^2(1-\lambda_n)^2}{4\lambda_nn}\right)
\end{align*}
and
\begin{align*}
\varphi_{\mathcal{G}_1}\left(-\frac{t\sqrt{\lambda_n}}{2\sqrt{n}}\right)=1-\frac{\Var|g_1|^2}{4}\frac{t^2\lambda_n}{2n}+o\left(\frac{t^2\lambda_n}{4n}\right)
\end{align*}
we have that if $\lambda_nn=k_n\to\infty$, as $n\to\infty$, for every $t\in\R$,
$$
 \lim_{n\to\infty}\varphi_{Y^{(1)}_n}(t)= e^{-\frac{t^2}{2}\eta^2},
$$
with the exponent $\eta^2$ given by
$$
\eta^2:=\frac{\lambda\Var|Z_1|^2}{4M_\infty^2(2)}+\frac{(1-\lambda)\Var|g_1|^2}{4}.
$$
Therefore, as $n\to\infty$, the random variable $Y^{(1)}_n$ converges in distribution to a centered Gaussian random variable with variance
$$
\eta^2=\frac{\lambda\Var|Z_1|^2}{4M_\infty^2(2)}+\frac{(1-\lambda)\Var|g_1|^2}{4}=\frac{1}{2}-\frac{3\lambda}{10}=\sigma^2(\infty,\lambda).
$$
By Slutsky's theorem (see Lemma \ref{thm:Slutsky}) all that is left to prove is that
\[
Y^{(3)}_n:=\sqrt{k_n}\,\Psi_\infty\,\Big(\frac{\xi_n^{(1)}}{\sqrt{n}},\frac{\xi_n^{(2)}}{\sqrt{k_n}},\frac{\xi_n^{(3)}}{\sqrt{n}}\Big)
\]
 converges to $0$ in probability. To this end, we observe that
\begin{eqnarray*}
Y^{(3)}_n&=&\sqrt{k_n}\,\bigg(\frac{(\xi_n^{(1)})^2}{n}+\frac{(\xi_n^{(2)})^2}{k_n}+\frac{(\xi_n^{(3)})^2}{n}\bigg)\frac{\Psi_\infty\Big(\frac{\xi_n^{(1)}}{\sqrt{n}},\frac{\xi_n^{(2)}}{\sqrt{k_n}},\frac{\xi_n^{(3)}}{\sqrt{n}}\Big)}{\Big\|\Big(\frac{\xi_n^{(1)}}{\sqrt{n}},\frac{\xi_n^{(2)}}{\sqrt{k_n}},\frac{\xi_n^{(3)}}{\sqrt{n}}\Big)\Big\|_2^2}\cr
&=&\bigg(\sqrt{\lambda_n}\xi_n^{(1)}\frac{\xi_n^{(1)}}{\sqrt{n}}+\xi_n^{(2)}\frac{\xi_n^{(2)}}{\sqrt{k_n}}+\sqrt{\lambda_n}\xi_n^{(3)}\frac{\xi_n^{(3)}}{\sqrt{n}}\bigg)\frac{\Psi_\infty\Big(\frac{\xi_n^{(1)}}{\sqrt{n}},\frac{\xi_n^{(2)}}{\sqrt{k_n}},\frac{\xi_n^{(3)}}{\sqrt{n}}\Big)}{\Big\|\Big(\frac{\xi_n^{(1)}}{\sqrt{n}},\frac{\xi_n^{(2)}}{\sqrt{k_n}},\frac{\xi_n^{(3)}}{\sqrt{n}}\Big)\Big\|_2^2}\,.
\end{eqnarray*}
Using the same argument as in the case where $p<\infty$, we obtain that this random variable converges to $0$ in probability, as $n\to\infty$. This finishes the proof of part (a) of Theorem \ref{thm:CLTEuclideanNormProjections infinity}.
\end{proof}

\section{Proof of the Berry-Esseen bounds}\label{sec:berry-esseen}

In this section we will present the proof of the Berry-Esseen bounds in Theorem \ref{thm:berry-esseen} if $p<\infty$ and Theorem \ref{thm:CLTEuclideanNormProjections infinity} if $p=\infty$. Recall that we need that the subspace dimensions grow fast enough with $n$. More precisely, we require that $\frac{k_n}{n^{2/3}}\to\infty$, as $n\to\infty$, and refer to Remark \ref{rem:Restriction} for a related discussion.

Postponing the proof for $p=\infty$ to Subsection \ref{sec:BerryEsseenp=Infinity}, we fix $1\leq p<\infty$ and recall the definitions of the random variables $Y^{(1)}_n, Y^{(2)}_n$, and $Y^{(3)}_n$ that were introduced in Section \ref{sec:clt}:
\begin{align*}
Y^{(1)}_n&=\frac{\sqrt{\lambda_n}\xi_n^{(1)}}{2M_p(2)} - \frac{\sqrt{\lambda_n}\xi_n^{(2)}}{p} + \frac{\xi_n^{(3)}}{2}-\frac{\sqrt{\lambda_n}\xi_n^{(4)}}{2},\cr
Y^{(2)}_n& =-\frac{\sqrt{\lambda_n}}{p\sqrt{n}}\,W, \cr
Y^{(3)}_n& =\sqrt{k_n}\Psi_p\left(\frac{\xi_n^{(1)}}{\sqrt{n}},\frac{\xi_n^{(2)}}{\sqrt{n}},\frac{\xi_n^{(3)}}{\sqrt{k_n}},\frac{\xi_n^{(4)}}{\sqrt{n}},\frac{W}{n}\right).
\end{align*}
The starting point for our proof is the following lemma which we will apply later with an $\varepsilon$ depending on the dimension parameter $n$ and to $Y_n^{(1)},Y_n^{(2)}$ and $Y_n^{(3)}$.

\begin{lemma}\label{lem:SplitProbabilities}
Let $Y_1,Y_2,Y_3$ be three random variables, let $G$ be a centered Gaussian random variable with variance $\sigma^2>0$ and let $\varepsilon>0$. Then
\begin{align*}
&\sup_{t\in\R}\big|\Pro\left(Y_1+Y_2+Y_3\geq t\right)-\Pro\left(G\geq t\right)\big|\\
&\quad\leq\sup_{x\in\R}\big|\Pro\left(Y_1\geq x\right)-\Pro\left(G\geq x\right)\big|+\Pro\left(|Y_2|>\frac{\varepsilon}{2}\right)+\Pro\left(|Y_3|>\frac{\varepsilon}{2}\right)+\frac{\varepsilon}{\sqrt{2\pi\sigma^2}}.
\end{align*}
\end{lemma}

\begin{proof}
For any $t\in\R$ we have that
\begin{align*}
\Pro\left(Y_1+Y_2+Y_3\geq t\right)&=\Pro\left(Y_1+Y_2+Y_3\geq t,|Y_2+Y_3|\leq\varepsilon\right)\cr
&\quad+\Pro\left(Y_1+Y_2+Y_3\geq t,|Y_2+Y_3|>\varepsilon\right).
\end{align*}
Therefore, using the mean value theorem,
\begin{align*}
&\Pro\left(Y_1+Y_2+Y_3\geq t\right)-\Pro\left(G\geq t\right)\\
&=\Pro\left(Y_1+Y_2+Y_3\geq t,|Y_2+Y_3|\leq\varepsilon\right)-\Pro\left(G\geq t\right)\cr
&\qquad\qquad+\Pro\left(Y_1+Y_2+Y_3\geq t,|Y_2+Y_3|>\varepsilon\right)\cr
&\leq\Pro\left(Y_1\geq t-\varepsilon\right)-\Pro\left(G\geq t\right)+\Pro\left(|Y_2+Y_3|>\varepsilon\right)\cr
&=\Pro\left(Y_1\geq t-\varepsilon\right)-\Pro\left(G\geq t-\varepsilon\right)\cr
&\qquad\qquad+\Pro\left(G\geq t-\varepsilon\right)-\Pro\left(G\geq t\right)+\Pro\left(|Y_2+Y_3|>\varepsilon\right)\cr
&\leq\left|\Pro\left(Y_1\geq t-\varepsilon\right)-\Pro\left(G\geq t-\varepsilon\right)\right|\cr
&\qquad\qquad+\left|\Pro\left(G\geq t-\varepsilon\right)-\Pro\left(G\geq t\right)\right|+\Pro\left(|Y_2+Y_3|>\varepsilon\right)\cr
&\leq\left|\Pro\left(Y_1\geq t-\varepsilon\right)-\Pro\left(G\geq t-\varepsilon\right)\right|\cr
&\qquad\qquad+\frac{\varepsilon}{\sqrt{2\pi\sigma^2}}+\Pro\left(|Y_2|>\frac{\varepsilon}{2}\right)+\Pro\left(|Y_3|>\frac{\varepsilon}{2}\right).
\end{align*}
In the same way,
\begin{align*}
&\Pro\left(Y_1+Y_2+Y_3\geq t\right)-\Pro\left(G\geq t\right)\\
&=\Pro\left(G< t\right)-\Pro\left(Y_1+Y_2+Y_3< t\right)\cr
&=\Pro\left(G<t\right)-\Pro\left(Y_1+Y_2+Y_3< t,|Y_2+Y_3|\leq\varepsilon\right)\cr
&\qquad\qquad-\Pro\left(Y_1+Y_2+Y_3< t,|Y_2+Y_3|>\varepsilon\right)\cr
&\geq\Pro\left(G< t\right)-\Pro\left(Y_1< t+\varepsilon\right)-\Pro\left(|Y_2+Y_3|>\varepsilon\right)\cr
&=\Pro\left(G< t\right)-\Pro\left(G< t+\varepsilon\right)\cr
&\qquad\qquad+\Pro\left(G< t+\varepsilon\right)-\Pro\left(Y_1< t+\varepsilon\right)-\Pro\left(|Y_2+Y_3|>\varepsilon\right)\cr
&\geq-\left|\Pro\left(G< t+\varepsilon\right)-\Pro\left(G<t\right)\right|\cr
&\qquad\qquad-\left|\Pro\left(G< t+\varepsilon\right)-\Pro\left(Y_1< t+\varepsilon\right)\right|-\Pro\left(|Y_2+Y_3|>\varepsilon\right)\cr
&\geq-\frac{\varepsilon}{\sqrt{2\pi \sigma^2}}-\left|\Pro\left(G\geq t+\varepsilon\right)-\Pro\left(Y_1\geq t+\varepsilon\right)\right|-\Pro\left(|Y_2|>\frac{\varepsilon}{2}\right)-\Pro\left(|Y_3|>\frac{\varepsilon}{2}\right).
\end{align*}
Putting together both inequalities and taking the supremum over all $t\in\R$ completes the proof.
\end{proof}

Our goal is to apply Lemma \ref{lem:SplitProbabilities} to the random variables $Y_n^{(1)}$, $Y_n^{(2)}$, and $Y_n^{(3)}$ and to estimate each one of the three terms separately. The terms that involve $Y_n^{(1)}$ and $Y_n^{(3)}$ will be handled in the next two subsections. The proof of Theorem \ref{thm:berry-esseen} will then be completed in Subsection \ref{sec:BerryEsseenComplete}.

\subsection{Part A - Estimate for $Y_n^{(1)}$}

In order to bound the first term we will use the following three lemmas. The first one is elementary and we refrain from giving a proof.

\begin{lemma}\label{lem:DifferenceComplexNumbers}
\begin{itemize}
\item[(a)] Let $w,z\in \C$ be such that $|w|\leq|z|$ and $n\in\N$. Then
$$
|w^n-z^n|\leq n|w-z||z|^{n-1}.
$$
\item[(b)] Let $y\in\R$. Then
$$
\left|e^{iy}-1-iy+\frac{y^2}{2}\right|\leq\frac{|y|^3}{6}.
$$
\end{itemize}
\end{lemma}


\begin{lemma}\label{lem:BoundsForCharacteristicFunctions}
Let $X$ be a centered random variable with characteristic function $\varphi_X$ and finite third moment. Then, for $t\in\R$,
$$
\left|\varphi_X(t)-1+\frac{t^2}{2}\Var X\right|\leq\frac{\E|X|^3|t|^3}{6}.
$$
Moreover, if $|t|\leq\min\left\{\frac{\sqrt2}{\sqrt{\Var X}},\frac{4\Var X}{3\E|X|^3}\right\}$,
$$
|\varphi_X(t)|\leq 1-\frac{5\Var X}{18}t^2\leq e^{-\frac{5\Var X}{18}t^2}.
$$
\end{lemma}

\begin{proof}
Applying Lemma \ref{lem:DifferenceComplexNumbers} (b) with $y=tX$ we obtain
$$
\left|e^{itX}-1-itX+\frac{t^2X^2}{2}\right|\leq\frac{|t|^3|X|^3}{6}.
$$
Taking expectations and using the triangle inequality together with the fact that $X$ is centered, we see that
$$
\left|\varphi_X(t)-1+\frac{t^2}{2}\Var X\right|\leq\E\left|e^{itX}-1-itX+\frac{t^2X^2}{2}\right| \leq\frac{\E|X|^3|t|^3}{6}.
$$
Consequently, by the triangle inequality
$$
\left|\varphi_X(t)\right|-\left|1-\frac{t^2}{2}\Var X\right|\leq\frac{\E|X|^3|t|^3}{6},
$$
which, if $|t|\leq\frac{\sqrt2}{\sqrt{\Var X}}$, gives
$$
\left|\varphi_X(t)\right|\leq 1-\frac{t^2}{2}\Var X+\frac{\E|X|^3|t|^3}{6}.
$$
If, in addition, $|t|\leq \frac{4\Var X}{3\E|X|^3}$, then we obtain
$$
\left|\varphi_X(t)\right|\leq 1-\frac{t^2}{2}\Var X+\frac{4t^2}{18}\Var X= 1-\frac{5t^2}{18}\Var X\leq e^{-\frac{5t^2}{18}\Var X}.
$$
This gives the desired bound.
\end{proof}

We will also need the following fact from complex analysis, see Theorem 3.1.1, Theorem 3.1.8 and Equation (3.1.7) in \cite{SimonComplexAnalysis}.

\begin{lemma}\label{taylor complex}
Let $f$ be a holomorphic function in an open disc $D(z,\delta)$ of radius $\delta>0$ around $z\in\mathbb{C}$. Then $f$ is analytic at $z$ and, for any $w\in D(z,\delta)$ and $n_0\in\N$,
\begin{align*}
f(w) = \sum_{n=0}^{n_0}a_n(w-z)^n + R_{n_0}(w),
\end{align*}
where, for any $0<\delta'<\delta$ and $n=0,\ldots,n_0$,
$$
a_n = {1\over 2\pi i}\oint_{C(z,\delta')}{f(\zeta)\over(\zeta-z)^{n+1}}\dif\zeta\quad\text{and}\quad |R_N(w)|\leq{|w|^{N+1}\over 1-|w|}\sup_{\zeta\in D(z,\delta')}|f(\zeta)|
$$
with $C(z,\delta')$ being the boundary of the disc $D(z,\delta')$.
\end{lemma}

In the following lemma we estimate the first term $Y_n^{(1)}$.

\begin{lemma}\label{lem:FirstTerm}
Let $Y^{(1)}_n$ be the random variable defined before and $N$ be a centered Gaussian random variable with variance $\sigma^2(p,\lambda)$, where $\sigma^2(p,\lambda)$ is the constant defined in Theorem \ref{thm:CLTEuclideanNormProjections}. Then, there exists a constant $C(p)\in(0,\infty)$ depending only on $p$ such that
$$
\sup_{t\in\R}\big|F_{Y_n^{(1)}}(t)-\Phi(t)\big|\leq C(p)\max\left\{\frac{1}{\lambda_n\sqrt{k_n}},|\lambda_n-\lambda|\right\},
$$
where $\Phi$ is the distribution function of $N$.
\end{lemma}

\begin{proof}
Recall the definitions of the random variables $\mathcal{Z}_i$ and $\mathcal{G}_i$ from \eqref{eq:DefZiGi}, let us introduce the shorthand notation
\begin{align*}
\sigma_1^2 &:= \Var\mathcal Z_1
=\frac{M_p(4)}{4M_p^2(2)}+\frac{M_p(2p)}{p^2M_p^2(p)}-\frac{M_p(p+2)}{p^2M_p(2)M_p(p)}-\frac{1}{4},\\
\rho_1 &:= \E\left|\mathcal Z_1\right|^3,\\
\sigma_2^2 &= \sigma_3^2:=\Var(\mathcal G_1)=M_2(4)-1=2,\\
\rho_2 &=\rho_3:=\E|\mathcal G_1|^3
\end{align*}
and note that
\begin{align}\label{eq:SigmaRepresenation1}
\sigma^2 = \lambda\sigma_1^2+{1-\lambda\over 4}\sigma_2^2= \lambda \sigma_1^2+\frac{(1-\lambda)\Var|g_1|^2}{4} = \lambda \sigma_1^2+\frac{(1-\lambda)}{2}.
\end{align}
Notice also that all these numbers are well-defined and do not depend on the parameter $n$. Let $\Theta:=\min\{\Theta_1,\Theta_2,\Theta_3,\Theta_4,\Theta_5\}$, where
\begin{align*}
\Theta_1 &:=\sqrt\frac{n}{\lambda_n}\min\left\{\frac{\sqrt{2}}{\sigma_1},\frac{4\sigma_1^2}{3\rho_1}\right\},\qquad\Theta_2:=\frac{2\sqrt{\lambda_n n}}{1-\lambda_n}\min\left\{\frac{\sqrt{2}}{\sigma_2},\frac{4\sigma_2^2}{3\rho_2}\right\},\\
\Theta_3 &:=2\sqrt\frac{n}{\lambda_n}\min\left\{\frac{\sqrt{2}}{\sigma_3},\frac{4\sigma_3^2}{3\rho_3}\right\},\quad\;\;\Theta_4:=\frac{\sqrt{k_n}}{4(1-\lambda_n)}\quad\text{and}\quad\Theta_5:=\frac{\sqrt{n}}{4\sqrt{\lambda_n}}.
\end{align*}
Observe that
\begin{align}\label{eq:BoundTheta}
\Theta \geq c_0(p) \sqrt{k_n},
\end{align}
where $c_0(p)\in(0,\infty)$ is a constant only depending on $p$ or, more precisely, on $\sigma_1$ and $\rho_1$.

From the very definition of $\Theta$, for any $t$ such that $|t|\leq\min\{\Theta_1,\Theta_2,\Theta_3\}$ we have, by Lemma \ref{lem:BoundsForCharacteristicFunctions}, that
\[
\left|\varphi_{\mathcal Z_1}\left(\frac{t\sqrt{\lambda_n}}{\sqrt n}\right)\right|\leq e^{-\frac{5\sigma_1^2\lambda_nt^2}{18n}}
\]
as well as
\[
\left|\varphi^{\lambda_n}_{\mathcal G_1}\left(\frac{t(1-\lambda_n)}{2\sqrt{\lambda_nn}}\right)\right|\leq e^{-\frac{5\sigma_2^2(1-\lambda_n)^2t^2}{72n}}=e^{-\frac{5(1-\lambda_n)^2t^2}{36n}}
\]
and
\[
\left|\varphi^{1-\lambda_n}_{\mathcal G_1}\left(-\frac{t\sqrt{\lambda_n}}{2\sqrt{n}}\right)\right|\leq e^{-\frac{5\sigma_3^2(1-\lambda_n)\lambda_nt^2}{72n}}=e^{-\frac{5(1-\lambda_n)\lambda_nt^2}{36n}}.
\]
Therefore, the product is bounded above by
\begin{align*}
\left|\varphi_{\mathcal Z_1}\left(\frac{t\sqrt{\lambda_n}}{\sqrt n}\right)\varphi^{\lambda_n}_{\mathcal G_1}\left(\frac{t(1-\lambda_n)}{2\sqrt{\lambda_nn}}\right)\varphi^{1-\lambda_n}_{\mathcal G_1}\left(-\frac{t\sqrt{\lambda_n}}{2\sqrt{n}}\right)\right|
&\leq\quad e^{-\frac{10\sigma_1^2\lambda_n+5(1-\lambda_n)}{18}\frac{t^2}{2n}}.
\end{align*}
Let us call $c_p:=\min\left\{\frac{10\sigma_1^2}{18},\frac{5}{18}, \sigma^2\right\}$. Notice that this value only depends on $p$ and not on $n$. Therefore,
\[
\max\left\{e^{-\frac{10\sigma_1^2\lambda_n+5(1-\lambda_n)}{18}\frac{t^2}{2n}},e^{-\frac{\sigma^2 t^2}{2n}}\right\}\leq e^{-\frac{c_pt^2}{2n}}
\]
and, by Lemma \ref{lem:DifferenceComplexNumbers}, we have that
\begin{align*}
&\left|\varphi_{Y^{(1)}_n}(t)-e^{-\frac{\sigma^2t^2}{2}}\right|\leq ne^{-\frac{c_pt^2(n-1)}{2n}}\\
&\quad\times \left|\varphi_{\mathcal Z_1}\left(\frac{t\sqrt{\lambda_n}}{\sqrt n}\right)\varphi^{\lambda_n}_{\mathcal G_1}\left(\frac{t(1-\lambda_n)}{2\sqrt{\lambda_nn}}\right)\varphi^{1-\lambda_n}_{\mathcal G_1}\left(-\frac{t\sqrt{\lambda_n}}{2\sqrt{n}}\right)
-e^{-\frac{\sigma^2t^2}{2n}}\right|.
\end{align*}
Using the triangle inequality we see that the last term is bounded by
\begin{equation}\label{eq:xxyyzz}
\begin{split}
&\left|\varphi_{\mathcal Z_1}\left(\frac{t\sqrt{\lambda_n}}{\sqrt n}\right)\varphi^{\lambda_n}_{\mathcal G_1}\left(\frac{t(1-\lambda_n)}{2\sqrt{\lambda_nn}}\right)\varphi^{1-\lambda_n}_{\mathcal G_1}\left(-\frac{t\sqrt{\lambda_n}}{2\sqrt{n}}\right)
-1+\frac{\sigma^2t^2}{2n}\right|\cr
&\qquad\qquad+\left|1-\frac{\sigma^2t^2}{2n}-e^{-\frac{\sigma^2t^2}{2n}}\right|.
\end{split}
\end{equation}
Applying the classical Lagrange remainder for the Taylor expansion of $e^{-x}$  and taking into account that $\frac{\sigma^2t^2}{2n}\geq0$ we see that
\[
\left|1-\frac{\sigma^2t^2}{2n}-e^{-\frac{\sigma^2t^2}{2n}}\right|\leq  \frac{\sigma^4t^4}{8n^2}.
\]
In order to bound the first summand in \eqref{eq:xxyyzz}, notice that, by \eqref{eq:CharacteristicChiSquared1}, for every $t\in\R$,
\[
\varphi_{\mathcal G_1}(t)=\frac{e^{-it}}{(1-2it)^{1/2}}.
\]
Therefore, for any $t\in\R$,
\[
\varphi^{\lambda_n}_{\mathcal G_1}(t)=\frac{e^{-\lambda_nit}}{(1-2it)^{\lambda_n/2}}\qquad\text{and}\qquad \varphi^{1-\lambda_n}_{\mathcal G_1}(t)=\frac{e^{-(1-\lambda_n)it}}{(1-2it)^{(1-\lambda_n)/2}}.
\]
For $z\in\C$, let us call
$$
f(z):=\frac{e^{-\frac{\lambda_n z}{2}}}{(1-z)^{\lambda_n/2}}\qquad\text{and}\qquad g(z):=\frac{e^{-\frac{(1-\lambda_n) z}{2}}}{(1-z)^{(1-\lambda_n)/2}},
$$
and observe that
\[
\varphi^{\lambda_n}_{\mathcal G_1}(t)=f(2it)\qquad\text{and}\qquad \varphi^{1-\lambda_n}_{\mathcal G_1}(t)=g(2it).
\]
By Lemma \ref{taylor complex}, we obtain
$$
f(z)=1+\frac{\lambda_n}{4}z^2+R_f(z)\quad\textrm{and}\quad g(z)=1+\frac{1-\lambda_n}{4}z^2+R_g(z)
$$
where $R_f,R_g$ are remainder terms satisfying
$$
|R_f(z)|\leq C^{\lambda_n}\frac{|z|^3}{1-|z|}\quad\textrm{and}\quad |R_g(z)|\leq C^{1-\lambda_n}\frac{|z|^3}{1-|z|}
$$
for every $z\in\C$ such that $|z|\leq\frac{1}{4}$, where $C:=\max\limits_{|w|=\frac{1}{2}}\Big|\frac{e^{-\frac{ w}{2}}}{(1-w)^{1/2}}\Big|$. Then,
\begin{align*}
\varphi^{\lambda_n}_{\mathcal G_1}\left(\frac{t(1-\lambda_n)}{2\sqrt{\lambda_nn}}\right)& =1-\frac{t^2(1-\lambda_n)^2}{4n}+R_f\left(\frac{t(1-\lambda_n)i}{\sqrt{\lambda_nn}}\right),\cr
\varphi^{1-\lambda_n}_{\mathcal G_1}\left(-\frac{t\sqrt{\lambda_n}}{2\sqrt{n}}\right)& =1-\frac{t^2\lambda_n(1-\lambda_n)}{4n}+R_g\left(\frac{-t\sqrt{\lambda_n}i}{\sqrt{n}}\right), \cr
\varphi_{\mathcal Z_1}\left(\frac{t\sqrt{\lambda_n}}{\sqrt n}\right)& =1-\frac{\sigma_1^2\lambda_nt^2}{2n} +R\left(\frac{t\sqrt{\lambda_n}}{\sqrt n}\right),
\end{align*}
where, by Lemma \ref{lem:BoundsForCharacteristicFunctions}, $|R(t)|\leq\frac{\rho_1|t|^3}{6}$. Besides, for the particular choices of $z$ above, Lemma \ref{taylor complex} yields that, for every $t\in\R$ with $|t|\leq\min\{\Theta_4,\Theta_5\}$,
\[
\left|R_f\left(\frac{t(1-\lambda_n)i}{\sqrt{\lambda_nn}}\right)\right| \leq  C^{\lambda_n}\frac{|t|^3(1-\lambda_n)^3}{k_n^{3/2}\left(1-|t|\frac{1-\lambda_n}{\sqrt{k_n}}\right)} \leq \frac{C_1|t|^3}{k_n^{3/2}}
\]
and
\[
\left|R_g\left(\frac{-t\sqrt{\lambda_n}i}{\sqrt{n}}\right)\right| \leq C^{1-\lambda_n}\frac{|t|^3\lambda_n^3}{n^{3/2}\left(1-|t|\frac{\sqrt{\lambda_n}}{\sqrt{n}}\right)} \leq \frac{C_1|t|^3}{n^{3/2}}\,,
\]
with $C_1=\max\{ C, 1 \}$ and where we used that $\lambda_n,1-\lambda_n\in[0,1]$ for any $n\in\N$.
Taking into account \eqref{eq:SigmaRepresenation1} we obtain
\begin{align}\label{eq:ProductOfCharFunctions}
& \varphi_{\mathcal Z_1}\left(\frac{t\sqrt{\lambda_n}}{\sqrt n}\right)\varphi^{\lambda_n}_{\mathcal G_1}\left(\frac{t(1-\lambda_n)}{2\sqrt{\lambda_nn}}\right)\varphi^{1-\lambda_n}_{\mathcal G_1}\left(-\frac{t\sqrt{\lambda_n}}{2\sqrt{n}}\right)
-1+\frac{\sigma^2t^2}{2n}\\
& \nonumber= \frac{t^2}{4n}(\lambda_n-\lambda)(1-2\sigma_1^2)+R_f\left(\frac{t(1-\lambda_n)i}{\sqrt{\lambda_nn}}\right) + R_g\left(\frac{-t\sqrt{\lambda_n}i}{\sqrt{n}}\right) \\
&\nonumber\qquad\qquad  + R\left(\frac{t\sqrt{\lambda_n}}{\sqrt n}\right) + \Delta_p(t,n)\,,
\end{align}
where $\Delta_p(t,n)$ is the sum the 20 remaining terms that arise from the multiplication. Using the triangle inequality and the previous estimates for $R_f$, $R_g$ and $R$, the absolute value of the above difference \eqref{eq:ProductOfCharFunctions} is bounded above by
\begin{align*}
\frac{t^2}{4n}|\lambda_n-\lambda|(1+2\sigma_1^2)+\frac{2C_1|t|^3}{k_n^{3/2}}
 + \frac{\rho_1|t|^3}{6n^{3/2}} + |\Delta_p(t,n)|,
\end{align*}
where again we used that $\lambda_n\in[0,1]$ and that $k_n\leq n$ for any $n\in\N$.

Let $\Phi$ be the distribution function of a centered Gaussian random variable with variance $\sigma^2(p,\lambda)$. Then, by the smoothing inequality (Lemma \ref{lem:SmoothingTheorem}) we have that, for any $t\in\R$,
\[
\Big|F_{Y_n^{(1)}}(t)-\Phi(t)\Big|\leq{1\over\pi}\int_{-\Theta}^\Theta\frac{|\varphi_{Y_n^{(1)}}(t)-e^{-\frac{\sigma^2t^2}{2}}|}{|t|}\dif t+\frac{24}{\pi\sqrt{2\pi\sigma^2}\,\Theta}.
\]
Notice that, by the previous estimates, the integrand is bounded by
\[
e^{-c_p\frac{t^2(n-1)}{2n}}\left( \frac{\sigma^4|t|^3}{8n}+\frac{|t|}{4}|\lambda_n-\lambda|(1+2\sigma_1^2)+\frac{2C_1t^2}{\lambda_n\sqrt{k_n}}+\frac{\rho_1 t^2}{6\sqrt{n}}+\frac{n}{|t|}|\Delta_p(t,n)|\right).
\]
Therefore, bounding the integral from above by the integral over all of $\R$, we obtain the following bound for the integral:
\begin{align*}
& C_p\left(\frac{1}{n}+|\lambda_n-\lambda|+ \frac{1}{\lambda_n\sqrt{k_n}}+\frac{1}{\sqrt{n}}+ \int_{-\Theta}^{\Theta}\frac{n}{|t|}e^{-c_p\frac{t^2(n-1)}{2n}}|\Delta_p(t,n)| \dif t\right) \cr
& \leq \widetilde{C}_p\left(|\lambda_n-\lambda|+ \frac{1}{\lambda_n\sqrt{k_n}}+ \int_{-\Theta}^{\Theta}\frac{n}{|t|}e^{-c_p\frac{t^2(n-1)}{2n}}|\Delta_p(t,n)| \dif t\right),
\end{align*}
where $C_p,\widetilde{C}_p\in(0,\infty)$ are constants only depending on $p$. The last integral is handled in the same way using triangle inequality on $|\Delta_p(t,n)|$ and gives terms of smaller order.

Observing that $\Theta\geq c_0(p)\sqrt{k_n}$ by \eqref{eq:BoundTheta}, we obtain the desired upper bound for the expression $|F_{Y_n^{(1)}}(t)-\Phi(t)|$.
\end{proof}

\subsection{Part B - Estimate for $Y_n^{(3)}$}

In the following lemma we estimate the third term involving $Y_n^{(3)}$. We will use the following well-known estimate. Namely, if $N$ is a standard Gaussian random variable we have that, for any $t>0$,
\begin{align}\label{eq:TailsGaussian}
\Pro\left(N\geq t\right)&\leq\frac{1}{\sqrt{2\pi}t}e^{-\frac{t^2}{2}}.
\end{align}

\begin{lemma}\label{lem:Term3}
Let $Y^{(3)}_n$ be defined as before. There are absolute constants $\alpha_1,\alpha_2, C\in(0,\infty)$ such that if $\frac{\alpha_1\log k_n}{\sqrt{k_n}}\leq\varepsilon_n\leq \alpha_2\sqrt{k_n}$, then
$$
\Pro\left(|Y^{(3)}_n|\geq\varepsilon_n\right)\leq \frac{C}{\sqrt{k_n\log k_n}}+\frac{Cp^{3-6/p}}{\sqrt{k_n}}+2\Pro\bigg(|W|>\sqrt{\frac{n\log k_n}{\lambda_n}}\,\bigg).
$$
\end{lemma}
\begin{proof}
We consider the function
\[
F:\R^5\to \R,\qquad {\bf x}\mapsto \frac{\Big(1+\frac{x_1}{M_p(2)}\Big)^{1/2}}{\Big(1+x_2+x_5\Big)^{1/p}}\,\frac{\big(1+x_3\big)^{1/2}}{\big(1+x_4\big)^{1/2}}
\]
and denote by $\Psi_p$ the Lagrange remainder of the first-order Taylor expansion of $F$ at zero. That is,
$$
\Psi_p({\bf x})=\frac{1}{2}\sum_{i_1,i_2=1}^5\frac{\partial^2 F}{\partial x_i\partial x_j}({\bf y})x_{i_1}x_{i_2}
$$
for ${\bf x}=(x_1,x_2,x_3,x_4,x_5)$ and where ${\bf y}\in\R^5$ is such that $\|{\bf y}\|_2\leq\|{\bf x}\|_2$.

Next, we notice that for all $p\geq 1$, $M_p(2)\geq 1/3=:c_1$. Besides, notice that since all the second partial derivatives are continuous on a Euclidean ball of radius $c_1$ centered at the origin, there exists a positive absolute constant $C_1>0$ such that if $\|{\bf x}\|_2\leq \frac{c_1}{2}$, then, for every $1\leq i_1,i_2\leq 5$, we have that $\big|\frac{\partial^2 F}{\partial x_i\partial x_j}({\bf y})\big|\leq C_1$ and consequently
$$
\left|\Psi_p({\bf x })\right|\leq C_1\sum_{i_1,i_2=1}^5|x_{i_1}||x_{i_2}|=C_1\left(\sum_{i=1}^5|x_i|\right)^2 = C_1\|{\bf x}\|_1^2.
$$

Let us recall the definition of the random variable $Y_n^{(3)}$:
\[
Y_n^{(3)}=\sqrt{k_n}\,\Psi_p\,\bigg(\frac{\xi_n^{(1)}}{\sqrt{n}},\frac{\xi_n^{(2)}}{\sqrt{n}},\frac{\xi_n^{(3)}}{\sqrt{k_n}},\frac{\xi_n^{(4)}}{\sqrt{n}},\frac{W}{n}\bigg).
\]
Therefore, by the previous remarks,
\begin{align*}
& \Pro\left(|Y^{(3)}_n|\geq\varepsilon_n\right)  = \Pro\left(\sqrt{k_n}\,\bigg|\Psi_p\,\bigg(\frac{\xi_n^{(1)}}{\sqrt{n}},\frac{\xi_n^{(2)}}{\sqrt{n}},\frac{\xi_n^{(3)}}{\sqrt{k_n}},\frac{\xi_n^{(4)}}{\sqrt{n}},\frac{W}{n}\bigg)\bigg| \geq\varepsilon_n\right)  \\
& = \Pro\left(\bigg\|\Big(\frac{\xi_n^{(1)}}{\sqrt{n}},\frac{\xi_n^{(2)}}{\sqrt{n}},\frac{\xi_n^{(3)}}{\sqrt{k_n}},\frac{\xi_n^{(4)}}{\sqrt{n}},\frac{W}{n}\Big)\bigg\|_1^2\,\bigg|\Psi_p\,\bigg(\frac{\xi_n^{(1)}}{\sqrt{n}},\frac{\xi_n^{(2)}}{\sqrt{n}},\frac{\xi_n^{(3)}}{\sqrt{k_n}},\frac{\xi_n^{(4)}}{\sqrt{n}},\frac{W}{n}\bigg)\bigg|  \right.\cr
& \qquad\qquad\qquad\left. \geq\frac{\varepsilon_n}{\sqrt{k_n}C_1}C_1\bigg\|\Big(\frac{\xi_n^{(1)}}{\sqrt{n}},\frac{\xi_n^{(2)}}{\sqrt{n}},\frac{\xi_n^{(3)}}{\sqrt{k_n}},\frac{\xi_n^{(4)}}{\sqrt{n}},\frac{W}{n}\Big)\bigg\|_1^2\right)\cr
& \leq \Pro\left( \bigg\|\Big(\frac{\xi_n^{(1)}}{\sqrt{n}},\frac{\xi_n^{(2)}}{\sqrt{n}},\frac{\xi_n^{(3)}}{\sqrt{k_n}},\frac{\xi_n^{(4)}}{\sqrt{n}},\frac{W}{n}\Big)\bigg\|_1^2 \geq \frac{\varepsilon_n}{\sqrt{k_n}C_1} \right) \cr
& \qquad\qquad + \Pro\left( \bigg|\Psi_p\,\bigg(\frac{\xi_n^{(1)}}{\sqrt{n}},\frac{\xi_n^{(2)}}{\sqrt{n}},\frac{\xi_n^{(3)}}{\sqrt{k_n}},\frac{\xi_n^{(4)}}{\sqrt{n}},\frac{W}{n}\bigg)\bigg| \right.\cr
& \qquad\qquad\qquad\qquad\qquad\left. \geq C_1 \bigg\|\Big(\frac{\xi_n^{(1)}}{\sqrt{n}},\frac{\xi_n^{(2)}}{\sqrt{n}},\frac{\xi_n^{(3)}}{\sqrt{k_n}},\frac{\xi_n^{(4)}}{\sqrt{n}},\frac{W}{n}\Big)\bigg\|_1^2 \right) \cr
& \leq \Pro\left( \frac{|\xi_n^{(1)}|}{\sqrt{n}}+\frac{|\xi_n^{(2)}|}{\sqrt{n}}+\frac{|\xi_n^{(3)}|}{\sqrt{k_n}}+\frac{|\xi_n^{(4)}|}{\sqrt{n}}+\frac{|W|}{n} \geq \frac{\sqrt{\varepsilon_n}}{k_n^{1/4}\sqrt{C_1}} \right)  \cr
& \qquad\qquad+ \Pro\left( \bigg\|\Big(\frac{\xi^{(1)}_n}{\sqrt{n}},\frac{\xi^{(2)}_n}{\sqrt{n}},\frac{\xi^{(3)}_n}{\sqrt{k_n}},\frac{\xi^{(4)}_n}{\sqrt{n}},\frac{W}{n}\Big) \bigg\|_2 > \frac{c_1}{2}\right) \cr
&\leq\Pro\left(|\xi^{(1)}_n|>\frac{c_1\sqrt{n}}{\sqrt{20}}\right)+\Pro\left(|\xi^{(2)}_n|>\frac{c_1\sqrt{n}}{\sqrt{20}}\right)+\Pro\left(|\xi^{(3)}_n|>\frac{c_1\sqrt{k_n}}{\sqrt{20}}\right)\cr
&\quad+\Pro\left(|\xi^{(4)}_n|>\frac{c_1\sqrt{n}}{\sqrt{20}}\right)+\Pro\left(|W|>\frac{c_1n}{\sqrt{20}}\right)\cr
&\quad+\Pro\left(|\xi^{(1)}_n|>\frac{\sqrt{\varepsilon_nn}}{5\sqrt{C_1}k_n^{1/4}}\right)+\Pro\left(|\xi^{(2)}_n|>\frac{\sqrt{\varepsilon_nn}}{5\sqrt{C_1}k_n^{1/4}}\right)+\Pro\left(|\xi^{(3)}_n|>\frac{\sqrt{\varepsilon_n}k_n^{1/4}}{5\sqrt{C_1}}\right)\cr
&\quad+\Pro\left(|\xi^{(4)}_n|>\frac{\sqrt{\varepsilon_nn}}{5\sqrt{C_1}k_n^{1/4}}\right)+\Pro\left(|W|>\frac{\sqrt{\varepsilon_n}n}{5\sqrt{C_1}k_n^{1/4}}\right).
\end{align*}
Besides, using the definition of $M_p(q)$, we obtain that
\[
\Var|Z_1|^2=\frac{2p^{4/p}}{p^2},\qquad \Var |Z_1|^p=p \qquad\text{and}\qquad \Var|g_1|^2=2.
\]
Thus, by the classical Berry-Esseen bound (Lemma \ref{lem:Berry-Esseen}), if $N\sim\mathcal{N}(0,1)$ is a standard Gaussian random variable, we obtain from the Gaussian tail estimate \eqref{eq:TailsGaussian} that
\begin{align*}
\Pro\left(|\xi^{(1)}_n|>\frac{c_1\sqrt{n}}{\sqrt{20}}\right)& \leq \Pro\left(|N|>\frac{c_1\sqrt{n}}{\sqrt{20\Var|Z_1|^2}}\right)+\frac{2\E\big||Z_1|^2-M_p(2)\big|^3}{(\Var|Z_1|^2)^{3/2}\sqrt{n}} \cr
& \leq\frac{\sqrt{2}\sqrt{20}p^{2/p}e^{-\frac{p^2c_1^2n}{80p^{4/p}}}}{pc_1\sqrt{2\pi n}}+\frac{2\E\big||Z_1|^2-M_p(2)\big|^3}{(\Var|Z_1|^2)^{3/2}\sqrt{n}}\,.
\end{align*}
Similarly, we have the following bounds:
\begin{align*}
\Pro\left(|\xi^{(2)}_n|>\frac{c_1\sqrt{n}}{\sqrt{20}}\right)&\leq\frac{\sqrt{p}\sqrt{20}e^{-\frac{c_1^2n}{40p}}}{c_1\sqrt{2\pi n}}+\frac{2\E\big||Z_1|^p-1\big|^3}{(\Var|Z_p|^p)^{3/2}\sqrt{n}}\,,\cr
\Pro\left(|\xi^{(3)}_n|>\frac{c_1\sqrt{k_n}}{\sqrt{20}}\right)& \leq\frac{\sqrt{2}\sqrt{20}e^{-\frac{c_1^2k_n}{80}}}{c_1\sqrt{2\pi k_n}}+\frac{2\E\big||g_1|^2-1\big|^3}{2^{3/2}\sqrt{k_n}}\,,\cr
\Pro\left(|\xi^{(4)}_n|>\frac{c_1\sqrt{n}}{\sqrt{20}}\right)& \leq\frac{\sqrt{2}\sqrt{20}e^{-\frac{c_1^2n}{80}}}{c_1\sqrt{2\pi n}}+\frac{2\E\big||g_1|^2-1\big|^3}{2^{3/2}\sqrt{n}}
\end{align*}
as well as
\begin{align*}
\Pro\left(|\xi^{(1)}_n|>\frac{\sqrt{\varepsilon_nn}}{5\sqrt{C_1}k_n^{1/4}}\right)& \leq\frac{5\sqrt{2}p^{2/p}\sqrt{C_1}k_n^{1/4}e^{-\frac{p^2\varepsilon_nn}{100p^{4/p}C_1\sqrt{k_n}}}}{p\sqrt{2\pi\varepsilon_nn}}\\
&\qquad\qquad+\frac{2\E\big||Z_1|^2-M_p(2)\big|^3}{(\Var|Z_1|^2)^{3/2}\sqrt{n}}\,,\cr
\Pro\left(|\xi^{(2)}_n|>\frac{\sqrt{\varepsilon_nn}}{5\sqrt{C_1}k_n^{1/4}}\right)& \leq\frac{5\sqrt{p}\sqrt{C_1}k_n^{1/4}e^{-\frac{\varepsilon_nn}{50pC_1\sqrt{k_n}}}}{\sqrt{2\pi\varepsilon_nn}}+\frac{2\E\big||Z_1|^p-M_p(p)\big|^3}{(\Var|Z_1|^p)^{3/2}\sqrt{n}}\,,\cr
\Pro\left(|\xi^{(3)}_n|>\frac{\sqrt{\varepsilon_n}k_n^{1/4}}{5\sqrt{C_1}}\right)& \leq\frac{5\sqrt{2}\sqrt{C_1}e^{-\frac{\varepsilon_n\sqrt{k_n}}{{100C_1}}}}{\sqrt{2\pi\varepsilon_n}k_n^{1/4}}+\frac{2\E\big||g_1|^2-1\big|^3}{2^{3/2}\sqrt{k_n}}\,,\cr
\Pro\left(|\xi^{(4)}_n|>\frac{\sqrt{\varepsilon_nn}}{5\sqrt{C_1}k_n^{1/4}}\right)&\leq\frac{5\sqrt{2}\sqrt{C_1}k_n^{1/4}e^{-\frac{\varepsilon_nn}{100C_1\sqrt{k_n}}}}{\sqrt{2\pi\varepsilon_nn}}+\frac{2\E\big||g_1|^2-1\big|^3}{2^{3/2}\sqrt{n}}.
\end{align*}
Using once more the definition of $M_p(q)$, we obtain that there exists an absolute constant $c_3\in(0,\infty)$ such that
\begin{align*}
\E\big||Z_1|^2-M_p(2)\big|^3& \leq M_p(6)+3M_p(4)M_p(2)+4M_p(2)^3\leq c_3,\cr
\E\big||Z_1|^p-1\big|^3& \leq M_p(3p)+3M_p(2p)+4\leq c_3p^2,\cr
\E\big||g_1|^2-1\big|^3& \leq M_2(6)+3M_2(4)+4\leq c_3.
\end{align*}
Therefore, if $\varepsilon_n\leq\frac{25C_1c_1^2\sqrt{k_n}}{20}$ we have that there exists an absolute constant $C_2\in(0,\infty)$ such that
\begin{eqnarray*}
&&\Pro\left(|Y^{(3)}_n|\geq\varepsilon_n\right)\leq\frac{40\sqrt{2}\sqrt{C_1}e^{-\frac{\varepsilon_n\sqrt{k_n}}{{100C_1}}}}{\sqrt{2\pi\varepsilon_n}k_n^{1/4}}+\frac{C_2p^{3-6/p}}{\sqrt{k_n}}+2\Pro\bigg(|W|>\frac{\sqrt{\varepsilon_n}n}{5\sqrt{C_1}k_n^{1/4}}\,\bigg).
\end{eqnarray*}
If we now take $\varepsilon_n\geq\frac{50C_1\log k_n}{\sqrt{k_n}}$, then we obtain that this quantity is bounded above by
$$
\frac{40\sqrt{2}}{\sqrt{100\pi k_n\log k_n}}+\frac{C_2p^{3-6/p}}{\sqrt{k_n}}+2\Pro\bigg(|W|>\sqrt{\frac{n\log k_n}{\lambda_n}}\,\bigg).
$$
This completes the proof.
\end{proof}

\subsection{Proof of the Berry-Esseen bound}\label{sec:BerryEsseenComplete}

We can now tie up all the loose ends developed so far and give a proof of our Berry-Esseen bound.

\begin{proof}[Proof of Theorem \ref{thm:berry-esseen}]
Let $1\leq p<\infty$ be fixed and let $N$ be a Gaussian random variable with variance $\sigma^2(p,\lambda)$ given by the expression in Theorem \ref{thm:CLTEuclideanNormProjections}. We apply Lemma \ref{lem:SplitProbabilities} to $Y_n^{(1)}$, $Y_n^{(2)}$ and $Y_n^{(3)}$ to conclude that
\begin{equation}\label{eq:Final1}
\begin{split}
&\big|\Pro\left(Y_n^{(1)}+Y_n^{(2)}+Y_n^{(3)}\geq t\right)-\Pro\left(N\geq t\right)\big|\leq\sup_{x\in\R}\big|\Pro\left(Y_n^{(1)}\geq x\right)-\Pro\left(N\geq x\right)\big|\cr
&\qquad\qquad\qquad+\Pro\left(|Y_n^{(2)}|>\frac{\varepsilon}{2}\right)+\Pro\left(|Y_n^{(3)}|>\frac{\varepsilon}{2}\right)+\frac{\varepsilon}{\sqrt{2\pi\sigma^2}}.
\end{split}
\end{equation}
From Lemma \ref{lem:FirstTerm} we have that
\begin{equation}\label{eq:Final2}
\begin{split}
\sup_{x\in\R}\big|\Pro\left(Y_n^{(1)}\geq x\right)-\Pro\left(N\geq x\right)\big| &\leq C(p)\max\left\{\frac{1}{\lambda_n\sqrt{k_n}},|\lambda_n-\lambda|\right\}\\
&= C(p)\max\left\{\frac{n}{k_n^{3/2}},\left|{k_n\over n}-\lambda\right|\right\}
\end{split}
\end{equation}
with a constant $C(p)\in(0,\infty)$ only depending on $p$. Next, we choose
$$
\varepsilon=\varepsilon_n={2\alpha\log k_n\over\sqrt{k_n}}
$$
for some absolute constant $\alpha\in(0,1)$ such that $2\alpha\geq\alpha_1$ with $\alpha_1$ being the constant from Lemma \ref{lem:Term3}. Then, by definition of $Y_n^{(2)}$,
\begin{align}\label{eq:Final3}
\Pro\left(|Y_n^{(2)}|>\frac{\varepsilon_n}{2}\right) = \Pro\left(|W|>{\alpha pn\log k_n\over\sqrt{k_n}}\right).
\end{align}
Moreover, Lemma \ref{lem:Term3} yields
\begin{equation}\label{eq:Final4}
\begin{split}
\Pro\left(|Y_n^{(3)}|>\frac{\varepsilon}{2}\right) &\leq \frac{C}{\sqrt{k_n\log k_n}}+\frac{Cp^{3-6/p}}{\sqrt{k_n}}+2\Pro\bigg(|W|>\sqrt{\frac{n\log k_n}{\lambda_n}}\,\bigg)\\
&= \frac{C}{\sqrt{k_n\log k_n}}+\frac{Cp^{3-6/p}}{\sqrt{k_n}}+2\Pro\bigg(|W|>\sqrt{\frac{n^2\log k_n}{k_n}}\,\bigg)
\end{split}
\end{equation}
with an absolute constant $C\in(0,\infty)$. Finally,
\begin{align}\label{eq:Final5}
\frac{\varepsilon_n}{\sqrt{2\pi\sigma^2}} = {2\alpha\log k_n\over \sqrt{2\pi\sigma^2\,k_n}} \leq \widetilde{C}(p)\,{\log k_n\over\sqrt{k_n}}
\end{align}
with another constant $\widetilde{C}(p)\in(0,\infty)$ only depending on $p$. Putting together \eqref{eq:Final1} with the estimates \eqref{eq:Final2}--\eqref{eq:Final5} and defining $C_p:=4\max\{C(p),\widetilde{C}(p),C,Cp^{3-6/p}\}$ yields the desired bound.
\end{proof}

\begin{rmk}\label{rem:Restriction}
It is evident from the proof that the argument in fact applies for any sequence $(k_n)_{n\in\N}$ which satisfies $k_n\to\infty$, as $n\to\infty$. However, the resulting bound is non-trivial if and only if in addition $k_n/n^{2/3}\to\infty$. This was our motivation to include this restriction already in the formulation of Theorem \ref{thm:berry-esseen}.
\end{rmk}

\begin{rmk}\label{rem:Constants}
An inspection of the above proof shows that the constant $C_p$ in Theorem \ref{thm:berry-esseen} satisfies
$$
\lim_{p\to 0}C_p = \lim_{p\to\infty}C_p=\infty.
$$
In particular, as $p\to\infty$, the constants $C_p$ explode. On the other hand, we still have a non-trivial Berry-Esseen bound with a finite absolute constant also in the case $p=\infty$, see Theorem \ref{thm:CLTEuclideanNormProjections infinity}. However, a separate proof for this case is needed.
\end{rmk}

\subsection{Handling the case $p=\infty$}\label{sec:BerryEsseenp=Infinity} The proof for the case $p=\infty$ is a line-by-line adaptation of the proof for $p<\infty$ and for this reason we decided to skip the details. In particular, working with uniformly distributed random variables on $[-1,1]$ instead of $p$-generalized Gaussian ones shows that the constant $C_\infty$ in Theorem \ref{thm:CLTEuclideanNormProjections infinity} (b) is finite. In view of Remark \ref{rem:Constants} this also shows that the constant $C_\infty$ cannot appear as the limit of the constants $C_p$ above, as $p\to\infty$.

\section{A large deviation principle}\label{sec:ldps}

As discussed in the introduction, a large deviation principle (LDP) for the Euclidean norms $\|P_{E_n}X_n\|_2$ was derived in \cite{APT16}. Here, $(E_n)_{n\in\N}$ is a sequence of $k_n$-dimensional random subspaces of $\R^n$, $1\leq k_n\leq n$, and $(X_n)_{n\in\N}$ is a sequence of uniformly distributed random points on $\B_p^n$ that are independent of the subspaces $E_n$. The purpose of this section is to present an extension of this LDP to the more general distributions $\bP_{n,p,\bW}$ if $p<\infty$ and to complement our central limit theorem (Theorem \ref{thm:CLTEuclideanNormProjections}). For that purpose, we first recall what it means that a sequence of random variables satisfies a large deviation principle. For further background material on large deviation theory we refer to \cite{DZ,dH} or \cite{Kallenberg}.

\begin{df}
Let $(Y_n)_{n\in\N}$ be a sequence of random variables. Further, let $s:\N\to[0,\infty]$ and $\mathcal{I}:\R\to[0,\infty]$ be a lower semi-continuous function with compact level sets $\{x\in\R\,:\, \mathcal{I}(x) \leq \alpha \}$, $\alpha\in\R$. We say that $(Y_n)_{n\in\N}$ satisfies a (full) large deviation principle (LDP) with speed $s(n)$ and (good) rate function $\mathcal{I}$ if
\begin{equation*}
\begin{split}
-\inf_{x\in A^\circ}\mathcal{I}(x) &\leq\liminf_{n\to\infty}{1\over s(n)}\log\Pro(Y_n\in A)\\
&\leq\limsup_{n\to\infty}{1\over s(n)}\log\Pro(Y_n\in A)\leq-\inf_{x\in\overline{A}}\mathcal{I}(x)
\end{split}
\end{equation*}
for all (Lebesgue-) measurable subsets $A\subseteq\R$, where $A^\circ$ and $\overline{A}$ stand for the interior and the closure of $A$, respectively.
\end{df}

To present the LDP for the Euclidean norms $\|P_{E_n}X_n\|_2$, we restrict to the case that $1\leq p<\infty$, since the uniform distribution on $\B_\infty^n$ has already been treated in \cite{APT16}. For $p\geq 2$ we introduce the function
\[
\mathcal{J}_p(y):=\inf_{x_1, x_2>0\atop {x_1^{1/2}}{x_2^{-{1/p}}}=y}\mathcal{I}_p^*(x_1,x_2)\,,\qquad y\in\R\,,
\]
where $\mathcal{I}_p^*(x_1,x_2)$ is the Legendre-Fenchel transform of
$$
\mathcal{I}_p(t_1,t_2):=\log\left(\int_{\R}e^{t_1x^2+t_2\vert x\vert^p}\,f_p(x)\,\dint x\right)\,,
$$
$(t_1,t_2)\in\R\times\big(-\infty,{1\over p}\big)$, with $f_p$ being the density of a $p$-generalized Gaussian random variable as in Proposition \ref{prop:schechtman zinn}.

\begin{thm}\label{thm:LDP}
Let $\bW$ be a probability distribution on $[0,\infty)$. Further, let $(k_n)_{n\in\N}$ be a sequence in $\N$ such that $1\leq k_n\leq n$ for each $n\in\N$, let $(X_n)_{n\in\N}$ be a sequence of random vectors $X_n$ distributed in $\B_p^n$ according to $\bP_{n,p,\bW}$ and let $(E_n)_{n\in\N}$ be a sequence of $k_n$-dimensional random subspaces $E_n$ of $\R^n$ that are distributed according to $\nu_{n,k_n}$. Assume that for each $n\in\N$, $X_n$ is independent of $E_n$. Let $W$ be a random variable with distribution $\bW$, which is independent of $(X_n)_{n\in\N}$ and $(E_n)_{n\in\N}$, and suppose that the sequence $(W/n)_{n\in\N}$ satisfies an LDP with speed $n$ and rate function $\mathcal{I}_{\bW}$.
\begin{itemize}
\item[(a)] Let $p\in[2,\infty]$ and assume that the limit $\lambda:=\lim\limits_{n\to\infty}{k_n\over n}$ exists in $[0,1]$. Then the sequence $\big(n^{{1\over p}-{1\over 2}}\|P_{E_n}X_n\|_2\big)_{n\in\N}$ satisfies an LDP with speed $n$ and a rate function $\mathcal{I}_1(y)$ that can be expressed in terms of $\lambda$, $\mathcal{I}_{\bW}$ and the function $\mathcal{J}_p$.

\item[(b)] Let $p\in[1,2)$ and assume that the limit $\lambda:=\lim\limits_{n\to\infty}{k_n\over n}$ exists in $(0,1]$. Assume that there is an open set $O\subset\R$ with $1\in O$ such that $\mathcal{I}_{\bW}(y)\neq 0$ for all $y\in O\setminus\{1\}$. Then the sequence $\big(n^{{1\over p}-{1\over 2}}\|P_{E_n}X_n\|_2\big)_{n\in\N}$ satisfies an LDP with speed $n^{p/2}$ and rate function
$$
\mathcal{I}_2(y):=\begin{cases}
{1\over p}\big({y^2\over\lambda}-m\big)^{p\over 2} &: y\geq \sqrt{\lambda\, m}\\
+\infty &: \text{otherwise}\,,
\end{cases}
$$
where $m=m_p:={p^{p/2}\over 3}{\Gamma(1+{3\over p})\over\Gamma(1+{1\over p})}$.
\end{itemize}
\end{thm}
\begin{proof}[Sketch of the proof of Theorem \ref{thm:LDP}, part (a)]
As the proof of Theorem 1.1 in \cite{APT16}, the argument is based on a repeated application of Cram\'er's theorem and the contraction principle. A glance at the probabilistic representation in Proposition \ref{prop:schechtman zinn} shows that the only essential difference to the set-up in \cite{APT16} is the new factor $\big(\sum_{n=1}^\infty|Z_i|^p+W\big)^{1/p}$. What remains to observe is that the random variables $\big({1\over n}\sum_{n=1}^\infty|Z_i|^p+{W\over n}\big)^{1/p}$
satisfy an LDP with speed $n$ and some rate function $\mathcal{I}_3$. In fact, this follows from Cram\'er's theorem, the contraction principle, the assumption that the random variables $W/n$ satisfy such an LDP and since $W$ and the random variables $Z_i$, $i\in\N$, are independent.
\end{proof}

\begin{proof}[Sketch of the proof of Theorem \ref{thm:LDP}, part (b)]
Again, the proof follows along the li\-nes of the proof of Theorem 1.2 in \cite{APT16} and we only indicate what needs to be adapted. In a first step, one shows that the sequence $\big(\sqrt{k_n\over n}\sum_{i=1}^nZ_i^2\big)_{n\in\N}$ satisfies an LDP with speed $n^{p/2}$ and rate function $\mathcal{I}_2$. The rest of the proof consists of showing that the sequences $\big(n^{{1\over p}-{1\over 2}}\|P_{E_n}X_n\|_2\big)_{n\in\N}$ and $\big(\sqrt{k_n\over n}\sum_{i=1}^nZ_i^2\big)_{n\in\N}$ are exponentially equivalent and thus satisfy the same LDP.  This is done in a way similar to that in \cite{APT16}. In fact, our assumption that the sequence $(W/n)_{n\in\N}$ satisfies an LDP with speed $n$ and a rate function $\mathcal{I}_{\bW}$ with $\mathcal{I}_{\bW}(y)\neq 0$ for all $y\in O\setminus\{1\}$ guarantees that the same arguments apply.
\end{proof}

Let us briefly check the assumption on $\bW$ in Theorem \ref{thm:LDP} for the distributions (i)--(iii)  introduced before Theorem \ref{thm:CLTEuclideanNormProjections}. If $\bW$ is the exponential distribution with mean $1$ or the gamma distribution with shape parameter $\alpha>0$ and rate $1$, one has that $W/n$ satisfies an LDP with speed $n$ and a linear rate function $\mathcal{I}_{\bW}(y)=y$ if $y\geq 0$ and $+\infty$ otherwise. These variables correspond to the uniform distribution as well as to a beta-type probability measure on $\B_p^n$. Finally, if $\bW$ is the Dirac measure at $0$, $W/n$ satisfies an LDP with speed $n$ and rate function $\mathcal{I}_{\bW}(y)\equiv+\infty$. As a consequence, the LDP in Theorem \ref{thm:LDP} applies to these situations.

\bibliographystyle{plain}
\bibliography{clt_ran_pro}

\end{document}